\documentclass[12pt]{article}

\usepackage[top=2cm, bottom=2cm, left=2cm, right=2cm] {geometry}
\usepackage{scalerel}
\usepackage{multicol}
\usepackage[breaklinks,colorlinks]{hyperref}
\usepackage{breakcites}
\usepackage[arrow,matrix]{xy}
\usepackage{graphicx}
\usepackage{stackrel}
\input xypic
\usepackage{amsfonts,amsmath,amssymb,amsthm,amscd,color,comment,tikz,pgfplots,lipsum,calc,hyperref}

\theoremstyle{plain}
\newtheorem{theorem}{Theorem}
\newtheorem{thm}[theorem]{Theorem}

\newtheorem{lem}[theorem]{Lemma}

\newtheorem{cor}[theorem]{Corollary}

\newtheorem{prop}[theorem]{Proposition}

\theoremstyle{remark}

\newtheorem{rem}[theorem]{Remark}

\theoremstyle{definition}
\newtheorem{defn}[theorem]{Definition}

\newtheorem{ex}[theorem]{Example}

\newcommand{\red}{\textcolor{red}}

\newcommand{\R}{\mathbb{R}}

\newcommand{\PP}{\mathbb{P}}
\renewcommand{\P}{\mathcal{P}}

\newcommand{\e}{\mathbf{e}}

\newcommand{\B}{\mathcal{B}}

\newcommand{\Z}{\mathbf{Z}}
\newcommand{\dv}{\mathbf{1}\hspace{-6 pt}\mathbf{1}}

\newcommand{\x}{\mathbf{x}}

\newcommand{\y}{\mathbf{y}}

\newcommand{\bel}[1]{\begin{equation}\label{#1}}

\newcommand{\be}{\begin{equation}}
\newcommand{\ba}{\begin{eqnarray}}
\newcommand{\ea}{\end{eqnarray}}

\newcommand{\qe}{\end{equation}}
\newcommand{\al}{\alpha}
\newcommand{\alv}{\boldsymbol{\alpha}}

\newcommand{\tev}{\boldsymbol{\theta}}

\newcommand{\de}{\delta}
\newcommand{\Om}{\Omega}
\newcommand{\De}{\Delta}

\newcommand{\suml}{\sum\limits}

\newcommand{\eps}{\varepsilon}

\begin{document}

\title{On the asymptotic behavior of the Diaconis and Freedman's chain in a multidimensional simplex}
\author{
Marc Peign\'e
\and
Tat Dat Tran
}

\newcommand{\Addresses}{{
  \bigskip
  \footnotesize

  \medskip
  
 Marc Peign\'e, \textsc{Institut Denis Poisson UMR 7013,  Universit\'e de Tours, Universit\'e d'Orl\'eans, CNRS  France. }\\ 
  \textit{E-mail address}: \texttt{peigne@univ-tours.fr}

  \medskip

  Tat Dat Tran, \textsc{Max-Planck-Institut f\"ur Mathematik in den Naturwissenschaften, Inselstrasse 22, D-04103 Leipzig, Germany\\
  Mathematisches Institut, Universität Leipzig, Augustusplatz 10, D-04109 Leipzig, Germany}\\
  \textit{E-mail address}:  \texttt{trandat@mis.mpg.de}, \texttt{tran@math-uni.leipzig.de}

}}

\date{\today}
\maketitle

\medskip

\abstract In this paper, we give out a setting of an Diaconis and Freedman's chain in a multidimensional simplex and consider its asymptotic behavior.  By using techniques in random iterated functions theory and quasi-compact operators theory, we first give out some  sufficient conditions which ensure the existence  and uniqueness of an invariant probability measure. In some particular cases, we  give out explicit formulas of the  invariant probability density. Moreover, we completely classify all behaviors of this chain in dimensional two. Eventually, some other settings of the chain are discussed. 

\endabstract

\maketitle



\tableofcontents


\section{Introduction}\label{sec:intro}

The main motivation in this paper is to propose a general setting for the so called ``Diaconis and Freedman's chain''  in $\mathbb R^d, d\geq 2$. First, we give out the most natural setting of this chain on a $d$-dimensional simplex and consider its asymptotic behavior by using   techniques from  random iterated functions theory and quasi-compact operators theory (see \cite{LP19} for using these techniques in dimensional one).  We have recently learnt that this multi-dimensional setting is also considered in \cite{MV2020} where the authors used another approach and consider only the cases of ergodicity.  Then, we  also discuss some other possible extensions.

Markov chains generated by products of independent random iterated functions have been the object of numerous works for more than 60 years. We refer to \cite{Harris52}, \cite{Bush1953}, \cite{Karlin1953} for first models designed for analyzing data in learning, \cite{Dubins1966}, \cite{GR86}, \cite{Letac86},
\cite{Mirek11} or \cite{Stenflo12} and references therein; see also \cite{Peigne11a} and \cite{Peigne11b} for such processes with weak contraction assumptions on the involved random functions.

In \cite{DF99}, Diaconis and Freedman focus on the Markov chain $(Z_n)_{n \geq 0}$ on $[0, 1]$  {randomly generated by the two families of maps $\mathcal H:= \{h_t: [0,1]\to [0,1], x \mapsto  tx \}_{t \in [0, 1]}$ and  $\mathcal A:= \{a_t: [0,1]\to [0,1], x\mapsto tx+1-t\}_{t \in [0, 1]}$; at each step, a map is randomly chosen with probability $p$ in the set $\mathcal H$ and $q=1-p$ in the set $\mathcal A$, then uniformly with respect to the parameter $t\in [0, 1]$. When the weight $p$ is constant, the random maps (see Section 3 for a detail introduction) which control the transitions of  this chain are i.i.d., otherwise the process $(Z_n)_{n\ge 0}$ is no longer in the framework of products of independent random functions. This class of such processes has been studied for a few decades, with various assumptions put on the state space (e.g. compactness) and the regularity of the weight functions. We refer to, for instance, \cite{Kaijser81}, \cite{BE88}, \cite{Barnsley88}, \cite{Barnsley89} with connections to image encoding a few years later, and \cite{KS17} more recently. All these works concern sufficient conditions for the unicity of the invariant measure and do not explore the case when there are several invariant measures. As far as we know, the coupling method does not seem to be relevant to study this type of Markov chains when there are further invariant measures, or, equivalently, when the space of harmonic functions is not reduced to constant.

For the Diaconis and Freedman' chain in  dimension 1, a systematic approach has been developed in \cite{LP19}, based on the theory of quasi-compact operators (also described in \cite{Peigne93} and \cite{HH01}); the authors describe completely the peripheral spectrum of the transition operator $P$  of $(Z_n)_{n \geq 0}$ and use  a precise control of the action of the family   of functions generated by the sets $\mathcal H$ and $\mathcal A$ according to $p$ and $P$. However, a multidimensional setting for such problems has not been touched;  it is our aim to introduce and analyse it here.

The paper is organized as follows. In Section 2 we give out our setting of the Diaconis and Freedman's chain in dimension $d \geq 2$.  Some properties of the transition operator and its dual operator have been considered and the uniqueness of the stationary density function has been shown (Corollary~\ref{cor:uniqueness}). In Section 3 we give out some results on uniqueness of invariant measures  (Theorems~\ref{thm:uniqueConstant} and ~\ref{thm:uniqueDepend}) based on concepts and results from the iterated functions system theory. Some special cases where we can find the explicit formula of the unique invariant density are  considered in Section 4. Section 5  contains our main result (Theorem~\ref{2DimGENERALDF}) where   we classify set of invariance probability measures and consider the asymptotic behavior of $(Z_n)_{n\ge 0}$. We discuss some future research directions in Section 6.

\section{The Diaconis and Freedman's chain  in dimension $\geq 2$}
In this section we consider a particular setting for the multi-dimensional problem of Diaconis and Freedman's chain. In fact, there are many ways to set which are based on different application models. Our setting here is fit for applications of robot controlling. Other interesting settings as well as their applications will be considered in details in somewhere else. 
Denote by 
\[
\De_d := \{ \x =(x_i)_{1\leq i \leq d} \in \R^d_{\ge 0} : |\x| = x_1+\cdots+x_{{d}} \le 1\} = co \{\e_0,\e_1,\ldots, \e_d\}
\]
a closed $d$-dimensional simplex with vertices $\e_0,\e_1,\ldots, \e_d$, where $\e_0 = (0,\ldots,0)$ and $\e_i=(0,\ldots,\underbrace{1}_{i^{th}},\ldots,0)$ for $1\leq i\leq d$. {From now on, for any $\x \in \De_d$, we set $x_0= 1-\vert \x\vert$; it holds $\x= x_0 \e_0+x_1 \e_1+\ldots + x_d \e_d$ with $x_i \geq 0$ and $x_0+\ldots +x_d=1$.}

We consider the Markov chain $(Z_n)_{n\ge 0}$ on  the simplex $\De_d$ corresponding to the successive positions  of a robot, according to the following rules:

- the  robot is put randomly at a point $Z_0$ in $\De_d$;

- if at time $n \ge0$, it is located at $Z_n=\x\in \De_d$, then it chooses the vertex $\e_i, 0\leq i\leq d$,  with probability $p_i(\x)$ for the next moving direction and uniformly randomly move to  some point on  the open line segment $(\x, \e_i):= \{t\x + (1-t)\e_i\mid  t\in (0,1)\}$.  

We assume that the functions $p_i, 0\leq i\leq d$, are continuous and non negative on $\De_d$ and satisfy $\sum_{i=0}^n p_i(\x) =1$ for any $\x \in \Delta_d$.

Let us make this description more rigorous. For any $ i=0, \ldots, d$ and $\x \in \Delta_d$,  denote by $\mu_i(\x,\cdot)$ the uniform distribution on $(\x, \e_i)$; it  is defined on open intervals $(\y_1,\y_2) \in \B((\x, \e_i))$ as
\bel{eq:1d}
\mu_i(\x, (\y_1,\y_2)) := |t(\x,\y_2,\e_i)-t(\x,\y_1,\e_i)|,
\qe
where the real number $t = t(\x,\y,\e_i)\in (0,1)$ solves the equality $\y = t \x + (1-t) \e_i$.
The {\it one-step transition probability function} $P$ of $(Z_n)_{n\ge 0}$ is 
\bel{eq:1}
P(\x,d\y)
= \sum_{i=0}^d p_i(\x) \mu_i(\x, d\y \cap (\x,\e_i)), \quad \x \in \De_d.
\qe
We   illustrate this setting in $\De_2$ in Figure~\ref{fig:2B}.
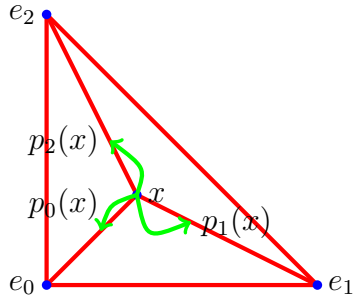
\begin{figure}[htb!]
\begin{center}
\begin{tikzpicture}[scale=1.2]
    \draw[color=red, ultra thick] (0,0) -- (3,0);
    \draw[color=red, ultra thick] (0,0) -- (0,3);
    \draw[color=red, ultra thick] (0,3) -- (3,0);
    \draw[color=red, ultra thick] (0,0) -- (1,1);
    \draw[color=red, ultra thick] (0,3) -- (1,1);
    \draw[color=red, ultra thick] (1,1) -- (3,0);
   \node[left] at (0,0) {$e_0$} ; 
   \node[right] at (3,0) {$e_1$} ; 
   \node[left] at (0,3) {$e_2$} ; 
   \node[below,right] at (1,1) {$x$} ; 
   \fill[blue] (0,0) circle (1.5pt);
   \fill[blue] (1,1) circle (1.5pt);   
   \fill[blue] (3,0) circle (1.5pt);
    \fill[blue] (0,3) circle (1.5pt);
    \draw[color=green, ultra thick, ->] (1,1) .. controls (0.7,.9) .. (0.6,0.6);
     \draw[color=green, ultra thick, ->] (1,1) .. controls (1.1,.5)  .. (1.6,.7);
      \draw[color=green, ultra thick, ->] (1,1) .. controls (1.1,1.3)  .. (0.7,1.6);
     \node[right] at (1.6,.7) {$p_1(x)$} ; 
      \node[left] at (0.7,1.6) {$p_2(x)$} ; 
            \node[left] at (0.7,0.9) {$p_0(x)$} ; 
        
         \end{tikzpicture}
\end{center}
\caption{The Diaconis and Freedman's chain in $\De_2$.}\label{fig:2B}
\end{figure}

We want  to classify the invariant probability measures of  the chain $(Z_n)_{n\ge 0}$ and to describe  its    behavior  as $n \to +\infty.$ Our approach is based on the description of the spectrum, on some suitable space to specify, of the   operator  corresponding to the  one-step transition probability function $P$, also denoted by $P$. Let us first introduce this transition operator.

We denote by $  \mathbb L^{\infty}(\De_d,d\x)$ the space of all bounded measurable functions $f:\De_d \to \mathbb C$ and  $\mathbb L^1(\De_d,d\y)$ the space of all integrable measurable functions $g: \De_d \to \mathbb C$; they are  Banach spaces,  endowed respectively with the norms $\Vert f\Vert_\infty := \sup_{\x\in\De_d} |f(\x)|$ and  $\Vert g\Vert_1 := \int_{\Delta_d}  |g(\y)|d\y$.

We also denote by $Den(\De_d,d\y)=\{g\in \mathbb L^1(\De_d,d\y): g\ge 0 \text{ and } \int_{\De_d}g(\y)d\y = 1\}$ the space of all probability densities on $\De_d$ with respect to the reference Lebesgue measure $d\y$. The set  $(Den(\De_d,d\y),d)$ is a complete metric space for  the distance $d(f,g) :=\Vert f-g\Vert_1$; furthermore, 
 $Den(\De_d,d\y)$ is a nonempty closed convex subset of the Banach space $\mathbb L^1(\De_d, d\y)$ and it contains the constant function $g(\y) \equiv  d !$. 
 
 We   drop the reference Lebesgue measure $d\x, d\y$ in our notations where no ambiguity arises.

The transition operator of the chain $(Z_n)_{n\ge 0}$ is defined by 
\begin{align}\label{operatorP}
P: \quad \mathbb L^{\infty}(\De_d) &\rightarrow \mathbb  L^{\infty}(\De_d)\\\notag
f \qquad &\mapsto 
\left(Pf: \x\to  \int_{\De_d} f(\y)P(\x,d\y)\right).
\end{align}
Its dual operator $P^* : \mathbb L^1(\De_d) \to \mathbb L^1(\De_d)$ is defined by
\begin{align}\label{operatorP*}
\int_{\De_d} Pf(\x)g(\x)d\x = \int_{\De_d} f(\x) P^*g(\x) d\x.
\end{align}
Let us explicit the form of these two operators.
\begin{lem} Let $P$ be the transition operator of $(Z_n)_{n\ge 0}$ and $P^*$ its dual operator. Then
\bel{defn:Q}
Pf(\x) = \sum_{i=0}^d p_i(\x) \int_0^1 f(t\x + (1-t)\e_i) dt
\qe
and  
\begin{equation}\label{defn:P*}
P^*g(\y)  = \sum_{i=0}^d \int_{1-y_i}^{1} t^{-d} G_i\Bigg(\frac{1}{t}\y + \Big(1-\frac{1}{t}\Big)\e_i\Bigg) dt
 =\sum_{i=0}^d \int_{1}^{\frac{1}{1-y_i}} s^{d-2} G_i\Bigg(s\y + \Big(1-s\Big)\e_i\Bigg) ds
\end{equation}
where $G_i(\y) = g(\y) p_i(\y)$.
\end{lem}
\begin{proof}Equality  Eq.~\eqref{eq:1} yields
\begin{align*}
Pf(\x) &= \int_{\De_d} f(\y)P(\x,d\y) = \sum_{i=0}^d p_i(\x) \int_{\De_d} f(\y) \mu_i(d\y \cap (\x,\e_i))\\
 &= \sum_{i=0}^d p_i(\x) \int_0^1 f(t\x + (1-t)\e_i) dt.
 \end{align*}
For the computation of $P^*$, we assume $d=2$; the same argument holds for any $d$. 
For all $f\in \mathbb  L^{\infty}(\De_d)$ and $g\in \mathbb L^1(\De_d)$, 
\begin{align*}
\int_{\De_2} f(\x) P^*g(\x) d\x &= \int_{\De_2} Pf(\x)g(\x) d\x \\ 
&=  \sum_{i=0}^2 \int_{\De_2}\Bigg( p_i(\x) \int_0^1 f(t\x + (1-t)\e_i) dt \Bigg) g(\x)  d\x\\
&=  \sum_{i=0}^2 \int_{\De_2}\Big( G_i(\x) \int_0^1 f(t\x + (1-t)\e_i) dt  \Big) d\x.
\end{align*}
Let us detail the computation of  the term $\int_{\De_2}\Big( G_0(\x) \int_0^1 f(t\x) dt  \Big) d\x$; the same calculation holds for the other terms. Namely, 
\begin{align*}
\int_{\De_2}\Big( G_0(\x) \int_0^1 f(t\x) dt  \Big) d\x &= \int_{0}^1 \Bigg[\int_0^{1-x_1} \Big( G_0(\x) \int_0^1 f(t\x) dt  \Big) dx_2 \Bigg] dx_1\\
&= \int_{0}^1 \Bigg[ \int_0^{1} \Big( \int_0^{1-x_1}  G_0(\x)f(t\x) dx_2  \Big) dt \Bigg] dx_1\\
&= \int_{0}^1 \Bigg[ \int_0^{1} \Big( \int_0^{1-x_1}  G_0(\x)f(t\x) dx_2  \Big) dx_1 \Bigg] dt\\
&\stackrel{y_2=tx_2}{=} \int_{0}^1 \Bigg[ \int_0^{1} \Big( \int_0^{(1-x_1)t}  G_0\Big(x_1,\frac{y_2}{t}\Big)f(tx_1, y_2) \frac{dy_2}{t}  \Big) dx_1 \Bigg] dt\\
&\stackrel{y_1=tx_1}{=} \int_{0}^1 \Bigg[ \int_0^{t} \Big( \int_0^{t-y_1}  G_0\Big(\frac{y_1}{t},\frac{y_2}{t}\Big)f(y_1, y_2) \frac{dy_2}{t}  \Big)  \frac{dy_1}{t}  \Bigg] dt\\
&= \int_{0}^1 \Bigg[ \int_{y_1}^{1} \Big( \int_0^{t-y_1} \frac{1}{t^2} G_0\Big(\frac{y_1}{t},\frac{y_2}{t}\Big)f(y_1, y_2) dy_2  \Big)  dt  \Bigg] dy_1\\
&= \int_{0}^1 \Bigg[ \int_{0}^{1-y_1} \Big( \int_{y_1+y_2}^{1} \frac{1}{t^2} G_0\Big(\frac{y_1}{t},\frac{y_2}{t}\Big)f(y_1, y_2) dt  \Big)  dy_2  \Bigg] dy_1\\
&= \int_{\De_2} f(\y) \Big( \int_{1-y_0}^{1} t^{-2} G_0\Big(\frac{1}{t} \y\Big) dt  \Big)  d\y.
\end{align*}
Similarly
$\displaystyle\int_{\De_2}\Big( G_i(\x) \int_0^1 f(t\x + (1-t)\e_i) dt  \Big) d\x=  \int_{1}^{\frac{1}{1-y_i}}   G_i\Bigg(s\y + \Big(1-s\Big)\e_i\Bigg) ds$ for   $i=1, 2$ and (\ref{defn:P*}) follows.
\end{proof}
 \begin{rem}
In  dimension $d=1$, this is thus the expression of $P^*$ given in  \cite{LP19}:
\[
P^*g(y)  = \int_{1-y}^{1} t^{-1} G_1\Bigg(\frac{1}{t}y + \Big(1-\frac{1}{t}\Big)\Bigg) dt+ \int_{y}^{1} t^{-1} G_0\Bigg(\frac{1}{t}y\Bigg) dt = \int_0^y \frac{G_1(s)}{1-s}ds + \int_{y}^1 \frac{G_0(s)}{s}ds.
\] 
\end{rem}
Let us summarize some simple properties of $P$ and $P^*$.
\begin{prop}\label{PP*}
\begin{enumerate}
\item The operator $P$ is a Markov operator  on $\mathbb L^\infty(\De_d,d\x)$, i.e.
 \begin{enumerate}
\item [(i)] $Pf \ge 0$ whenever $f\in \mathbb L^\infty(\De_d,d\x)$ and $f\ge 0$; 
\item [(ii)] $P1=1$.
\end{enumerate}
In particular, $\Vert Pf\Vert_\infty \le \Vert f\Vert_\infty$ for any $f\in \mathbb L^\infty(\De_d,d\x)$.
Furthermore, $P$ is a Feller operator on $\De_d$,  i.e. $Pf \in C(\De_d)$ for all $f\in C(\De_d)$.
\item $P^*$ acts on $\mathbb L^1(\De_d,d\y)$ and, for any  non negative function  $g\in \mathbb L^1(\De_d,d\y)$, 
\[ P^*g \ge 0\quad \text{and} \quad  
\Vert P^*g\Vert _1 =\Vert g\Vert _1.
 \]
Furthermore, $P^*$ acts on $Den(\De_d,d\y)$, i.e., $P^*: Den(\De_d,d\y) \to Den(\De_d,d\y)$ and, for all  $ g_1\ne g_2\in Den(\De_d,d\y)$, 
\begin{equation}\label{contractionstricte}
\Vert P^*g_1-P^*g_2\Vert_1  <\Vert g_1-g_2\Vert_1.
\end{equation}
\end{enumerate}
\end{prop}
\begin{proof} The properties of $P$ are quite obvious; in particular, the fact that  $P$ is a Feller operator is easily checked from  the representation~\eqref{defn:Q} of $P$.
 Similarly, the first properties of $P^*$ follow from the definition.
 
To establish (\ref{contractionstricte}), we first recall that $|P^*h|\leq  P^*\vert h|$ for any $h\in \mathbb L^1(\De_d, d\y)$, which yields 
\[
\Vert P^*h\Vert_1 \leq\Vert (P^*\vert f\vert) \Vert _1=\Vert h\Vert_1.
\]
More precisely,
\begin{align*}
|P^*h| = (P^*h)_+ + (P^*h)_{-} &=   \max\{0,P^*h\} + \max\{0, -P^*h\}\\
&=  \max\{0,P^*h_+ - P^*h_-\} + \max\{0, P^*h_- - P^*h_+\}\\
&\le \max\{0,P^*h_+ \} + \max\{0, P^*h_- \}= P^*h_+ +P^*h_- = P^*|h \vert;
\end{align*}
hence, equality $\Vert P^*h\Vert_1 =\Vert h\Vert_1 $ holds if and only if  $P^*h_-\equiv 0$ and $P^*h_+\equiv 0$.

Now, we fix $ g_1\ne g_2\in Den(\De_d,d\y)$ and set 
 $h=g_1-g_2 \not\equiv 0$; it holds $\Vert P^*g_1-P^*g_2\Vert_1  \leq \Vert g_1-g_2\Vert_1$.  If	 $\Vert P^*g_1-P^*g_2\Vert_1 =\Vert g_1-g_2\Vert_1$ then $P^*h_-=P^*h_+\equiv 0$ i.e. $ P^*\vert h\vert  \equiv 0$;  therefore $ \Vert  h \Vert_1 =\Vert( P^* \vert h\vert) \Vert_1 =0$, so that $h\equiv 0$, contradiction.
\end{proof} 
As a direct consequence of (\ref{contractionstricte}), we may state the following corollary.
\begin{cor}[Uniqueness of the stationary density function]\label{cor:uniqueness}
If there exists a stationary density function for the Markov chain $(Z_n)_{n\ge 0}$ then it is unique.
\end{cor}

\begin{proof}
Assume that there are two different stationary density functions $f\ne g\in Den(\De_d,d\y)$, i.e. $P^* f = f$ and $P^*g = g$. This   implies  $d(P^*f,P^*g) = d(f,g)$,  contradiction  with (\ref{contractionstricte}). 
\end{proof}

\begin{rem}
\begin{enumerate}
\item Although $(Den(\De_d,d\y),d)$ is a complete metric space,  the operator $P^*$ is not uniformly contractive, i.e. there exists $q\in [0,1)$ such that 
\[
d(P^* f, P^*g) \le q d(f,g) \quad \forall f,g \in Den(\De_d,d\y)
\]
therefore we can not apply the Banach fixed point theorem. In \cite[Proposition 2, p. 988-989]{RL10}, the authors applied the Banach fixed point theorem to prove the existence of the stationary density function but their argument does not work. A precise proof can be found in \cite[Theorem 3.1]{LP19} which covered all cases of $p_i(\x)$ in dimension 1.
\item Although $Den(\De_d,d\y)$ is a nonempty closed convex subset in a Banach space $\mathbb L^1(\De_d, d\x)$, we can not apply the Browder fixed point theorem, because  $\mathbb L^1(\De_d, d\x)$ is not  uniformly convex.
\item There are many cases of $p_i(\x)$ such that there is no stationary density function for the $(Z_n)_{n\ge 0}$ even in dimension 1: see cases 2 and 3 in \cite[Theorem 3.1]{LP19} where the set of invariant probability measures consist of convex combinations of Dirac measures $\de_0$ and $\de_1$. It will be interesting to classify cases of $p_i(\x)$ so that there exists (unique) a stationary density function. This is still an open question (see the last section of the present paper).   
\end{enumerate}
\end{rem}

\section{Uniqueness of invariant probability measure}
In this section, we recall some concepts as well as well-known results of iterated function systems and apply them to our model.

\subsection{Iterated function systems with place independent probabilities}

Let $(E, d)$ be  a compact metric space and denote $\displaystyle \mathbb L {\rm ip}(E, E) $ the space of Lipschitz continuous functions from  $E$  to $E$, i.e. of functions $T: E\to E$ such that 
$$ [T]:= \sup_{\stackrel{x, y
\in E}{x\neq y}} {d( T(x),T(y)) \over d(x,y)} < \infty.
$$
Let $(T_n)_{n \geq 1}$ be a sequence of i.i.d. random functions  defined on a probability space
 $(\Omega, \mathcal T, \mathbb P)$,   with values in $\mathbb L {\rm ip}(E, E)$
 and common distribution $\mu$. We consider the Markov chain $(X_n)_{n \geq 0} $ on $E$,
 defined by:  for any $n\geq 0$,
 \bel{eq:ifs}
 X_{n+1}:= T_{n+1}(X_n),
 \qe
 where $X_0$ is a fixed random variable  with values in $E$. One says that
  the chain $(X_n)_{n \geq 0} $ is
generated by   the {\it iterated function system}  $(T_n)_{n \geq
1}$. Its transition operator $Q$  is defined by: for any bounded
Borel
 function  $\varphi: E \to \mathbb C$
 and   any  $x\in E$
$$ Q\varphi(x) := \int_{\mathbb L {\rm ip}(E, E)} \varphi(T(x)) \mu({\rm d}T).$$
 The chain  $(X_n)_{n \geq 0} $ has the ``Feller property'',
i.e.  the operator $Q$ acts on the space  $C(E)$ of continuous
functions from $E$ to $\mathbb C$. The maps $T_n$ being Lipschitz
continuous on $E$,  the operator \red{$Q$} acts also on the space of
Lipschitz continuous functions from  $E$  to $\mathbb C$ and more generally
on the space  $ \mathcal   H_\alpha (E), 0<\alpha \leq 1$, of
$\alpha$-H\"older continuous functions from $E$ to $\mathbb C$,
defined by
$$
\mathcal   H_\alpha (E):= \{f\in C(E)\mid \Vert f\Vert_\alpha:= \Vert f\Vert_\infty +m_\alpha(f) <+\infty\}
$$
where $\displaystyle m_\alpha(f):=
\sup_{\stackrel{x, y \in E}{x\neq y}} {\vert f(x)-f(y)\vert \over
d(x,y)^\alpha} < \infty. $ 
Endowed, with the norm $\Vert \cdot
\Vert_\alpha$, the space $\mathcal   H_\alpha (E)$ is a Banach
space.

 The  behavior of the chain  $(X_n)_{n \geq 0} $ is closely related to
 the spectrum of the restriction of $Q$ to  these spaces.  Under some  ``contraction  in mean'' assumption
 on the  $T_n$, the restriction of $Q$ to  $\mathcal   H_\alpha (E)$ satisfies
 some spectral  gap property.
We first cite  the  following result in \cite[Proposition 2.1]{LP19}.

\begin{thm}[\cite{LP19}]\label{muconstant}
Assume that there   exists  $\alpha \in (0, 1]$
 such that
\begin{equation} \label{contraction}
r:= \sup_{\stackrel{x, y \in E}{x\neq y}}\int_{\mathbb L{\rm ip}(E, E)}\Bigl({ d(T(x),T(y)) \over d(x, y)}\Bigr)^\alpha \mu({\rm d}T) <1.
\end{equation}
Then, there exists on $E$ a unique $Q$-invariant probability measure    $\nu$.  Furthermore, there exists
 constants $\kappa >0$ and $\rho \in (0, 1)$    such that
\begin{equation} \label{vitesseExp}
\forall \varphi\in  \mathcal   H_\alpha (E),  \ \forall x \in E \quad  \vert Q^n\varphi(x)-\nu(\varphi)\vert \leq \kappa \rho^n.
\end{equation}
\end{thm}

\noindent {\bf Application to the Diaconis and Freedman's chain for $p$ fixed in $\De_d.$}
We assume  $p_i(\x) = p_i$ for all $i=0,\ldots, d$. We put the Diaconis and Freedman's chain into the framework of  iterated random functions as follows.
For each $i=0,\ldots, d$ and $t\in [0,1]$, we set $H_i(t,\cdot): \De_d \to \De_d, \x \mapsto t\x + (1-t)\e_i$ the affine transformation; these functions $H_i(t, \cdot)$ belong to    the space $\mathbb L {\rm ip}(\De_d, \De_d)$ of Lipschitz continuous functions from $\De_d$ to $\De_d$, with Lipschitz coefficient $m(H_i(t, \cdot))=t$.  Then, we consider the probability measure $\mu$ on $\mathbb L {\rm ip}(\De_d, \De_d)$ defined   by
\bel{eq:im1}
\mu(dT) := \sum_{i=0}^d p_i \int_{0}^1 \de_{H_i(t,\cdot)}(dT) dt, 
\qe
where $\de_T$ is the Dirac mass at $T$. Eq.~\eqref{defn:Q} may be rewritten as
\[
\forall f\in \mathbb  L^{\infty}(\De_d,d\x), \forall \x \in \De_d, \quad Q f(\x) = \int_{\mathbb L {\rm ip}(\De_d, \De_d)} f(T(\x))\mu(dT).
\]
Hence, the Diaconis and Freedman's chain $(Z_n)_{n \geq 0}$ on $\De_d$  is generated by the {\it iterated function system}  $(T_n)_{n \geq 1}$ in the sense of Eq.~\eqref{eq:ifs}, where $(T_n)_{n \geq 1}$ be a sequence of i.i.d. random functions with  common distribution $\mu$ defined by Eq.~\eqref{eq:im1}.
 
\begin{thm}\label{thm:uniqueConstant}
If $p_i(\x) = p_i$ for all $i=0,\ldots, d$,  then the Diaconis and Freedman's chain in $\De_d$ admits  a unique $P$-invariant probability measure $\nu \in \P(\De_d)$. Furthermore, there exists constants $\kappa >0$ and $\rho \in (0, 1)$ such that
\begin{equation}
\forall \varphi\in  \mathcal   H_\alpha (\De_d),  \ \forall \x \in \De_d \quad  \vert Q^n\varphi(\x)-\nu(\varphi)\vert \leq \kappa \rho^n.
\end{equation}

\end{thm} 
\begin{proof} This is a direct consequence of Theorem \ref{muconstant} with 
\begin{align*}
r=\sup_{\stackrel{\x, \y \in \De_d}{\x\neq \y}}\int_{\mathbb L {\rm ip}(\De_d, \De_d)}\Bigl({ |T(\x)-T(\y)| \over |\x-
\y|}\Bigr)^\alpha \mu({\rm d}T)   
&\leq  \sum_{i=0}^d p_i \int_0^1  m(H_i(t,\cdot)) ^\alpha
{\rm d}t\\
&=
 \int_0^1 t^\alpha {\rm d}t={1\over 1+\alpha} <1.
\end{align*} 
\end{proof}

\begin{rem}
The  unique  $P$-invariant probability measure $\nu$ is  usually  nothing but the Dirichlet distribution $Dir[ \theta_0,\ldots, \theta_k ]$ as will be shown later. If $\ \theta_i >0$ for all $i=0,\ldots, k$ we have a unique invariant probability density which is the Dirichlet density. If else, the Dirichlet distribution is singular and can be understood in the sense of \cite[p. 211]{Ferguson73}, \cite[Definition 3.1.1, p. 89]{GR03}, or \cite[Definition 4.2]{JLLT19}.
\end{rem}

\subsection{Iterated function systems with  place dependent probabilities}

In this subsection, we  extend the measure $\mu$ to a collection
$(\mu_x)_{x \in E}$  of probability measures on $E$, depending
continuously on $x$. We consider the Markov chain $(X_n)_{n \geq
0}$ on $E$ whose  transition operator $Q$  is given by: for any
bounded  Borel function  $\varphi: E \to \mathbb C$
 and  any  $x\in E$,
$$
Q\varphi(x)= \int_{\mathbb L {\rm ip}(E, E)} \varphi(T(x)) \mu_x({\rm d}T).
$$
 First, we introduce  the following definition.

\begin{defn}
  A sequence  $(\xi_n)_{n \geq 0}$ of  continuous functions from $E$ to $E$ is  a contracting sequence  if there exist  $x_0 \in E$ such that
  $$
\forall x \in E \quad \lim_{n \to +\infty} \xi_n(x)= x_0.$$
\end{defn}
We cite the  following result in \cite[Proposition 2.2]{LP19}. 
\begin{thm}[\cite{LP19}]\label{muvarie}
Assume that there   exists $\alpha \in (0,1]$ such that
\begin{enumerate}
\item [{\bf H1.}]  $r:=\displaystyle \sup_{\stackrel{x, y \in E}{x\neq y}}\int_{\mathbb L{\rm ip}(E, E)}\Bigl({ d(T(x), T(y))\over d(x, y)}\Bigr)^\alpha\mu_x({\rm d}T) <1;
 $
\item [{\bf H2.}]   $R_\alpha:= \displaystyle
\sup_{\stackrel{x, y \in E}{x\neq y}}{\vert \mu_x-\mu_y\vert_{ TV } \over d (x, y)^\alpha}<+\infty,  \text{ where $\vert \mu_x-\mu_y\vert_{TV}$ is the total variation  distance between $\mu_x$ and $\mu_y$};
$
\item [
{\bf H3.}] There exist  $\delta >0$ and a  probability measure
$\mu$ on  $E$  such that
\begin{enumerate}
\item[(i)]
$\label{minoration}
\forall x \in E\qquad \mu_x\geq \delta \mu;
$
\item[(ii)] the closed  semi-group   $T_\mu$  generated by the support
  $S_\mu$  of $ \mu$ possesses a contracting sequence.
  \end{enumerate}
\end{enumerate}
Then, there exists   on $E$ a unique   $Q$-invariant probability measure    $\nu$; furthermore, for some constants $\kappa >0$ and $\rho \in (0,1)$, it holds
\begin{equation}\label{vitesseExpbis}
\forall \varphi \in \mathcal   H_\alpha(E), \ \forall x \in E \quad \vert Q^n\varphi(x)-\nu(\varphi)\vert \leq \kappa\rho^n.
\end{equation}
\end{thm}

{Let us now apply this statement to the Diaconis and Freedman's chain on $\Delta_d$:  for each $\x\in \De_d$, we define a space-dependent probability measure $\mu_{\x}\in \P(X)$ by
\[
\mu_{\x}(dT) := \sum_{i=0}^d p_i(\x) \int_{0}^1 \de_{H_i(t,\cdot)}(dT) dt, 
\]
where $\de_T$ is the Dirac mass at $T$. With this collection $(\mu_{\x})_{{\bf x} \in \Delta_d}$ of probability measures,  the Diaconis and Freedman's chain falls within the scope of iterated function systems with spacial dependent increments probabilities.  }
\begin{thm}\label{thm:uniqueDepend}
{Assume that 

(1)  for all $j=0,\ldots, d$,  the  functions $p_j$  belong to  $\mathcal   H_\alpha (\De_d)$;

(2)  there exists $i\in \{0,\ldots, d\}$ such that $\delta_i:=\min_{\x \in \De_d}p_i(\x) > 0$.}

\noindent Then, the Diaconis and Freedman's chain in $\De_d$ has a unique $P$-invariant probability measure $\nu \in \P(\De_d)$. Furthermore, there exist constants $\kappa >0$ and $\rho \in (0, 1)$ such that
\begin{equation}
\forall \varphi\in  \mathcal   H_\alpha (\De_d),  \ \forall \x \in \De_d \quad  \vert P^n\varphi(\x)-\nu(\varphi)\vert \leq \kappa \rho^n.
\end{equation}
\end{thm} 
\begin{proof}
{This is a direct consequence of   Theorem~\ref{muvarie} since   conditions $H1.-H3.$ hold in this context. }
\begin{enumerate}
\item[H1.] For any $\x \ne \y \in \De_d$: 
\[
\int_{\mathbb L {\rm ip}(\De_d, \De_d)} \Bigg(\frac{|T(\x)-T(\y)|}{|\x-\y|}\Bigg)^{\al} \mu_{\x}(dT) = \sum_{i=0}^d p_i(\x) \int_0^1  \Big(\frac{H_i(t,\x)-H_i(t,\y)}{\x-\y}\Big)^{\al} dt = \frac{1}{1+\al} < 1;
\] 
\item[H2.] For any $\x \ne \y \in \De_d$ and { any Borel set $A \subseteq \Delta_d$} , 
\[
\frac{|\mu_{\x}(A)-\mu_{\y}(A)|}{|\x-\y|^{\al}} \le  \sum_{i=0}^d \frac{|p_i(\x)-p_i(\y)|}{|\x-\y|^{\al}} \int_0^1 H_i(t,\x)(A) dt \le \sum_{i=0}^d m_{\al}(p_i) < \infty.
\] 
\item[H3.] Set $\mu(dT) := \int_0^1 \de_{H_i(t,\cdot)}(dT) dt \in \P(X)$; it holds  $\mu_{\x}(dT) \ge p_i(\x) \int_{0}^1 \de_{H_i(t,\cdot)}(dT) dt \ge \de \mu(dT)$ for all $\x\in \De_d$. Moreover, the constant function $\x\mapsto 0$ belongs to the support of $\mu$ so that the semigroup $T_{\mu}$ contains a contracting sequence with limit point $0$. 
\end{enumerate}
\end{proof}

\section{Some explicit invariant probability densities}
In this section we consider some special cases of weights    for which it is possible  to compute explicitly  the unique invariant probability density. {When ${d}=1$, it has been known that, when both conditions  $p_1(0) > 0$ and $p_0(1)>0$ hold,    there exists a unique invariant probability density of $(Z_n)_{n\ge 0}$ given by
\[
g_{\infty}(y) = C\exp \Bigg(\int_{1/2}^y \frac{p_1(t)}{1-t} dt - \int_{1/2}^y \frac{p_0(t)}{t} dt \Bigg).
\]  
See for instance \cite{RL10}
 or  \cite{LP19}.}
 We do not get  such general result when $d \geq 2$, we can do it only in some specific cases.   We would also like to emphasize that in \cite{MV2020}, based on Sethuraman's construction of the Dirichlet distributions (see, \cite{Sethuraman1994}), the authors also gave out general results of the explicit formula of the stationary density in these special cases. Our approach is, however, very naturally and worth to be taken into account. 
\subsection{The case of constant weights}
We first consider the case of constant weights, i.e., $p_i(\x) = p_i>0$ for all $i=0,\ldots, d$. 
\begin{thm}\label{thm:constant}
If $p_i(\x)=p_i>0$ for all $i=0,\ldots,d$ then the unique invariant probability density $ g_\infty $ of $(Z_n)_{n\ge 0}$ is the density  of the Dirichlet distribution $Dir[p_0,\ldots,p_d]$, i.e.
\begin{equation}\label{dirichlet}
g_\infty (\y) = \frac{1}{\prod_{i=0}^d \Gamma(p_i)} \prod_{i=0}^d y_i^{p_i-1}  \dv_{\mathring \De_d }(\y)
\end{equation}
where $y_0:=1-y_1-\cdots-y_d$.
\end{thm}\
\begin{proof}
 {It suffices to prove } that $g_\infty(\y)=\prod_{i=0}^d y_i^{\al_i}$ with $\al_i=p_i-1$ is the unique solution of the equation $P^*g(\y) = g(\y)$. Indeed,
 \begin{align}\label{zlehf}
P^*g_\infty(\y) &= \sum_{i=0}^d p_i \int_{1}^{\frac{1}{1-y_i}} s^{d-2} g_\infty\Bigg(s\y + \Big(1-s\Big)\e_i\Bigg) ds\notag \\
&= \sum_{i=0}^d p_i \int_{1}^{\frac{1}{1-y_i}} s^{d-2} (sy_i+1-s)^{\al_i} \prod_{j\ne i}(sy_j)^{\al_j} ds\notag\\
&= \sum_{i=0}^d p_i \int_{1}^{\frac{1}{1-y_i}} s^{d-2+\sum_{j\ne i}\al_j} (sy_i+1-s)^{\al_i}  \frac{g_\infty(\y)}{y_i^{\al_i}}ds\notag \\
&=  \sum_{i=0}^d {p_i \over(1-y_i) y_i^{\al_i}(1-y_i)^{-2-\al_i}} \left(\int_{0}^{y_i}t^{\al_i}(1-t)^{-2-\al_i}dt \right) g_\infty(\y )
\end{align}
{  where the  last equality  follows by using the change of variables $t = sy_i+1-s$ and the equality  $\sum_{j=0}^d \al_j =
  -d$. 
  Notice that,   for  $i=0, \ldots, d$ and $\al_i>-1$,
\[
(1+\al_i)\int_{0}^{y_i}t^{\al_i}(1-t)^{-2-\al_i} dt = y_i^{1+\al_i}(1-y_i)^{-1-\al_i}.
\]
}
As a matter of fact,   the function $ F: y_i\mapsto (1+\al_i)\int_{0}^{y_i}t^{\al_i}(1-t)^{-2-\al_i} dt - y_i^{1+\al_i}(1-y_i)^{-1-\al_i}$ satisfies  $F(0)=0$ and $F'(y_i) = 0$ ; therefore $F(y_i) \equiv 0$. 
Hence, equality (\ref{zlehf}) yields 
\[
P^*g_\infty(\y) 
 = \sum_{i=0}^d  y_i g_\infty(\y) =g_\infty(\y).
\]
The uniqueness stems  from Corollary~\ref{cor:uniqueness}.
\end{proof}

\begin{ex}
When $d=2$, $p_1= p_2= p_0=1/3$, the unique invariant probability density is  \[
g_{\infty}(\y) = \frac{1}{\Gamma(1/3)^3} y_1^{-\frac{2}{3}} y_2^{-\frac{2}{3}}(1-y_1-y_2)^{-\frac{2}{3}} \dv_{\De_2^o}(\y). 
\qquad \text{(see Figure~\ref{Fig:DirichletDistribution})}\]

\begin{figure}[h!]
\begin{center}
	\includegraphics[width=45mm]{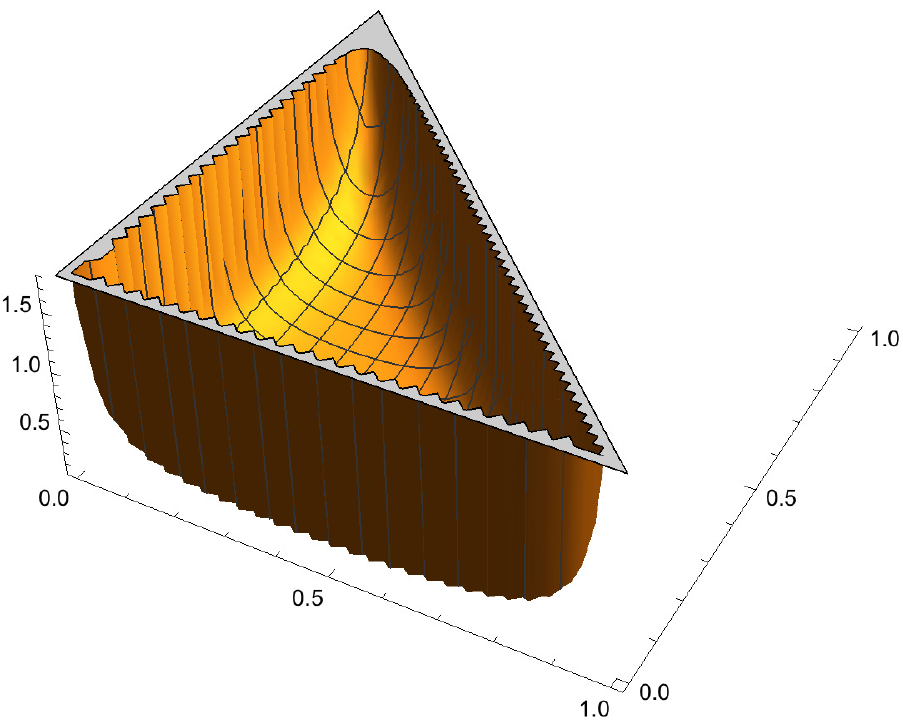}
	\end{center}
\caption{The density function of the Dirichlet distribution $Dir[\frac{1}{3},\frac{1}{3},\frac{1}{3}]$.}
\label{Fig:DirichletDistribution}
\end{figure}
\end{ex}
\noindent Figure~\ref{Fig:1} represents random trajectories of $(Z_n)_{n\ge 0}$

- in $[0,1]$,  starting at $x_0=0.6$ with $p_1(x)=0.2, p_0(\x) = 0.8$, in the left panel;

 - in $\De_2$,  starting at $\x_0=(0.3,0.4)$ with $p_1(\x) = 0.5 , p_2(\x)= 0.2, p_0(\x)=0.3$ in the right panel. 
 
 \noindent They both illustrate the ergodic behavior of the chain. 
\begin{figure}
\begin{center}
	\includegraphics[width=50mm]{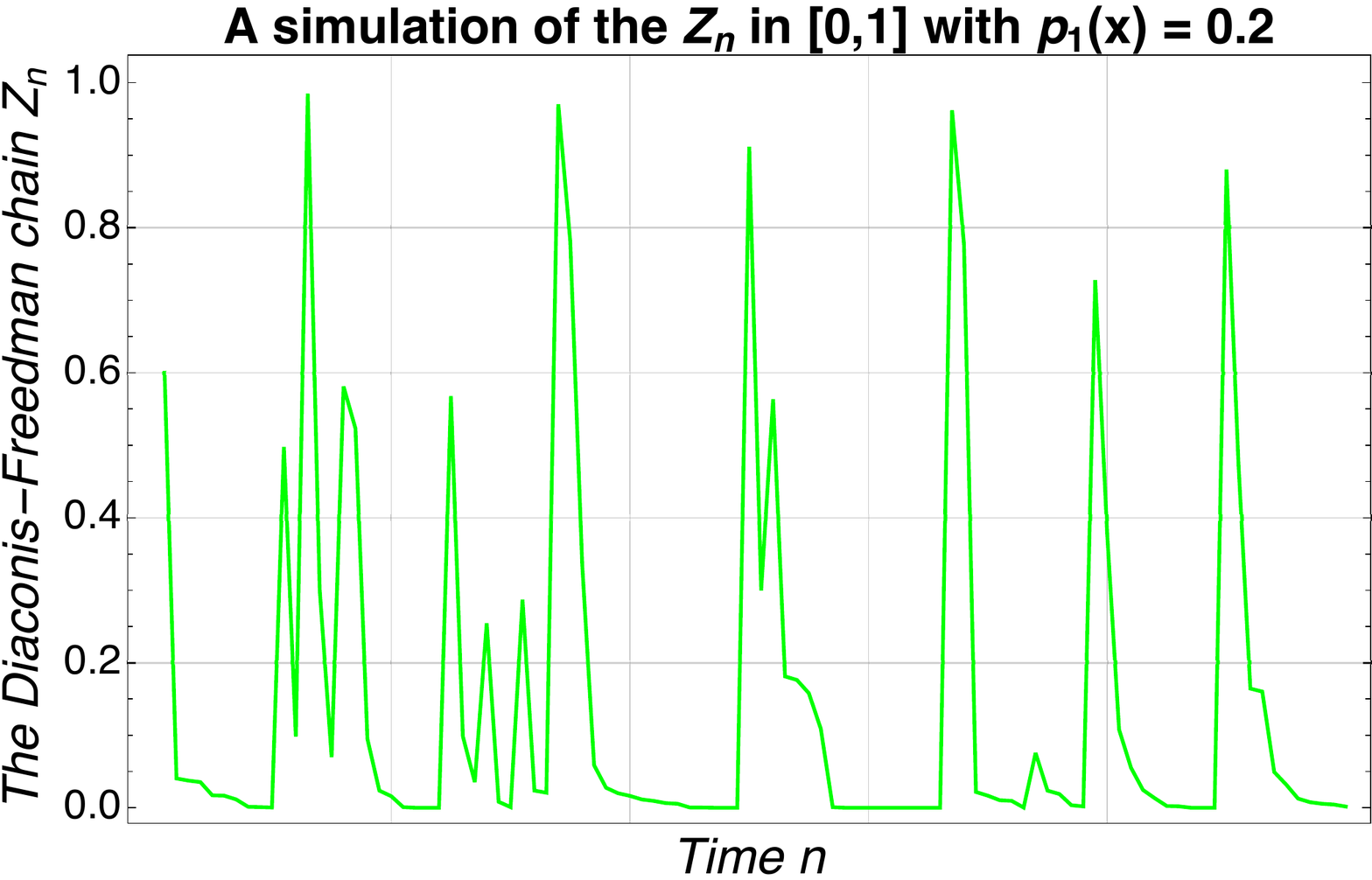}
	 \qquad  \qquad  \qquad 
	 \includegraphics[width=50mm]{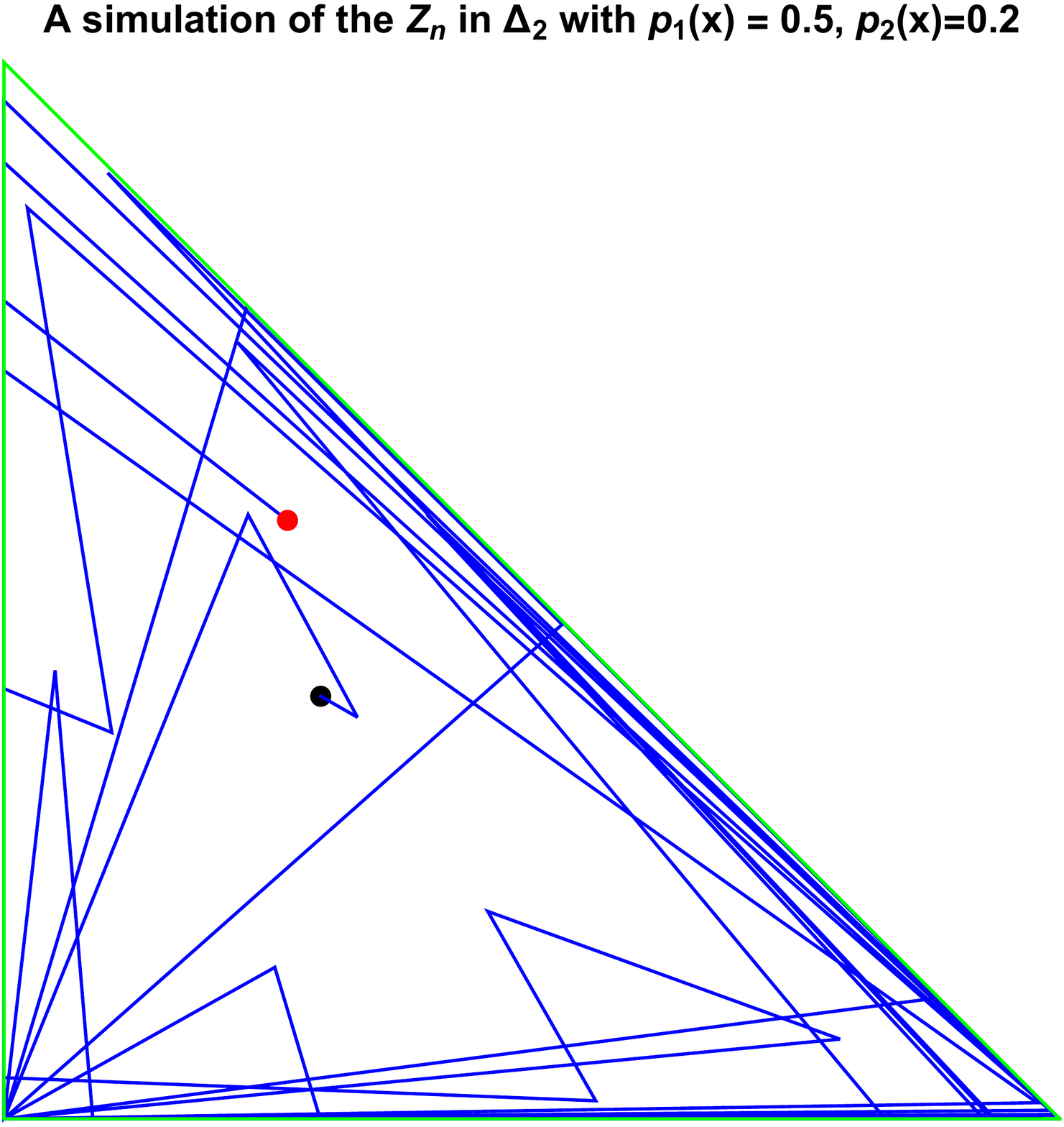}
	\end{center}
\caption{Left: Random trajectory of $Z_n$ in $[0,1]$ starting at $x_0=0.6$;
Right: Random trajectory of $Z_n$ in $\De_2$ starting at $\x_0=(0.3,0.4)$.}
\label{Fig:1}
\end{figure}

\newpage

\subsection{The case of affine  weights}
In this section, we present a  special case  of non constant weight functions $p_i(\y)$ which yield to the explicit form of the unique  invariant density function.
\begin{thm}\label{thm:linear}
Fix positive constants $\theta_0, \ldots, \theta_d  > 0$ with $|\tev| = \theta_0 +\cdots +\theta_d \le 1$ and assume that, for  any $i=1, \ldots, d$ and   ${\bf y}= (y_1, \ldots, y_d)$ in $\Delta_d$, 
\[p_i(\y) =\tilde p_i(y_i):= \theta_i +(1-|\tev| ) y_i\]
 (which implicitly implies $p_0(\y) =  \tilde p_0(y_0):=\theta_0 + (1-|\tev|) y_0$).
 Then, the unique invariant probability density $g_\infty$ of $(Z_n)_{n\ge 0}$ is the Dirichlet distribution $Dir[\theta_0,\ldots,\theta_d]$ given by  
\[
g_\infty(\y) = Dir[\theta_0,\ldots,\theta_d](\y) = \frac{\Gamma(|\tev|)}{\prod_{i=0}^d \Gamma(\theta_i)}\prod_{i=0}^d y_i^{\theta_i-1}\dv_{\mathring{\De}_d}(\y).
\]
\end{thm}
\begin{proof}
Using the same techniques as in the proof of Theorem~\ref{thm:constant}, we only need to check that
\[
\frac{\int_0^{y_i} p_i(t) t^{\al_i} (1-t)^{|\alv| -\al_i +d-2}dt }{y_i^{\al_i} (1-y_i)^{|\alv| -\al_i +d-1}} = y_i,
\]
which can be   easily done by a direct  calculation. It completes the proof.
\end{proof}
\begin{rem}
\begin{enumerate}
\item[(i)] The case when  $|\tev|=1$  corresponds to constant weights.
\item[(ii)] This result has a very closed connection to results studied in Wright-Fisher models with mutations (see  for instance \cite{THJ15a}, \cite{THJ15b}, \cite{HJT2017}).
\end{enumerate}
\end{rem}

In Figure~\ref{Fig:2} we simulate random trajectories of 
 $(Z_n)_{n\ge 0}$ 
  
  - in $[0,1]$,  starting at $x_0=0.6$ with $p_1(x)=x, p_0(x) = 1-x$ in the left panel;
  
- in $\De_2$ starting at $\x_0=(0.3,0.4)$ with $p_1(\x) = x_1 , p_2(\x)= x_2, p_0(\x)=1-x_1-x_2$ in the right panel.  They both illustrate the absorbing behavior of the chain.
\begin{figure}[h!]
\begin{center}
	\includegraphics[width=6cm]{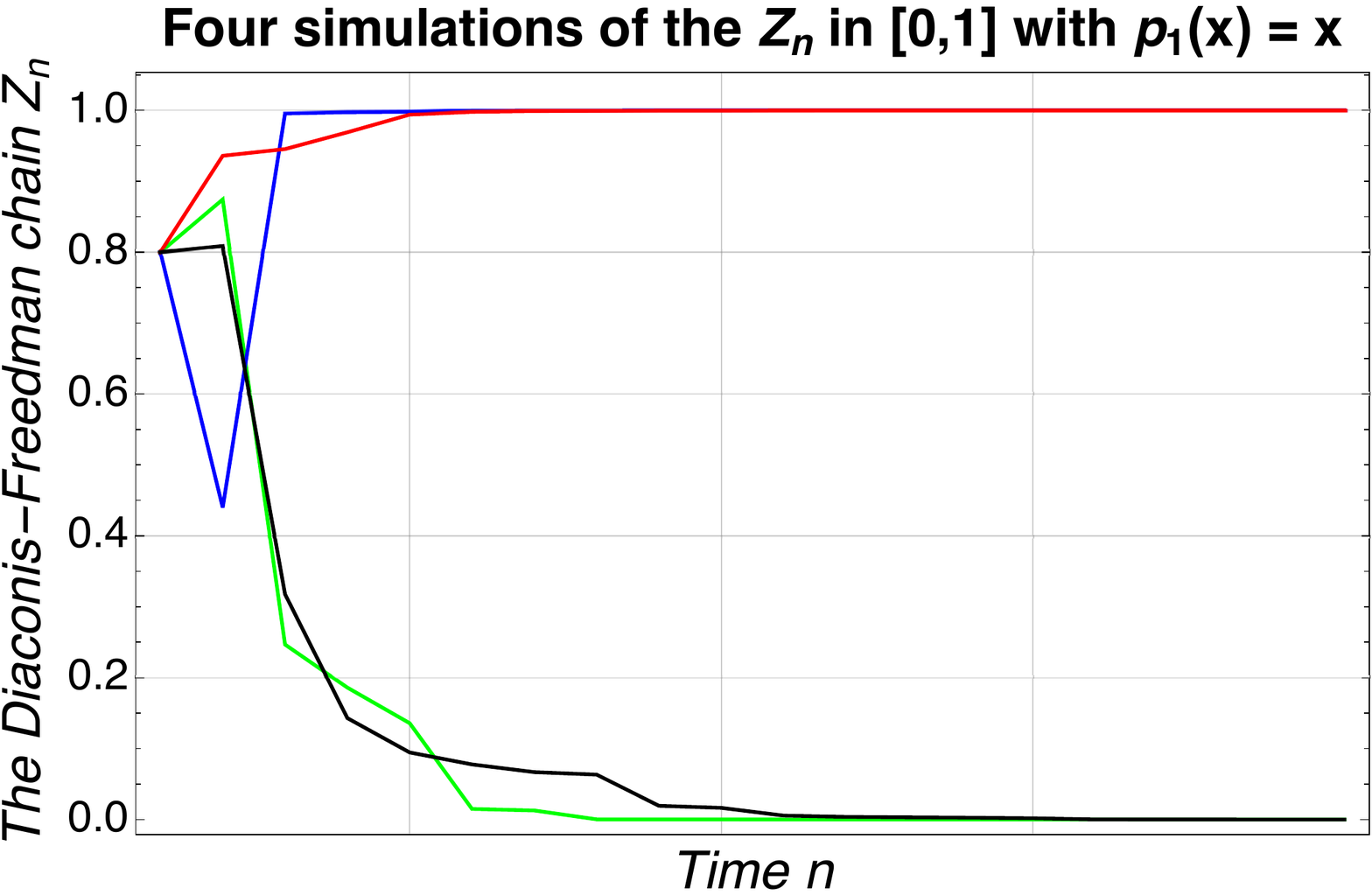} \qquad \includegraphics[width=6cm]{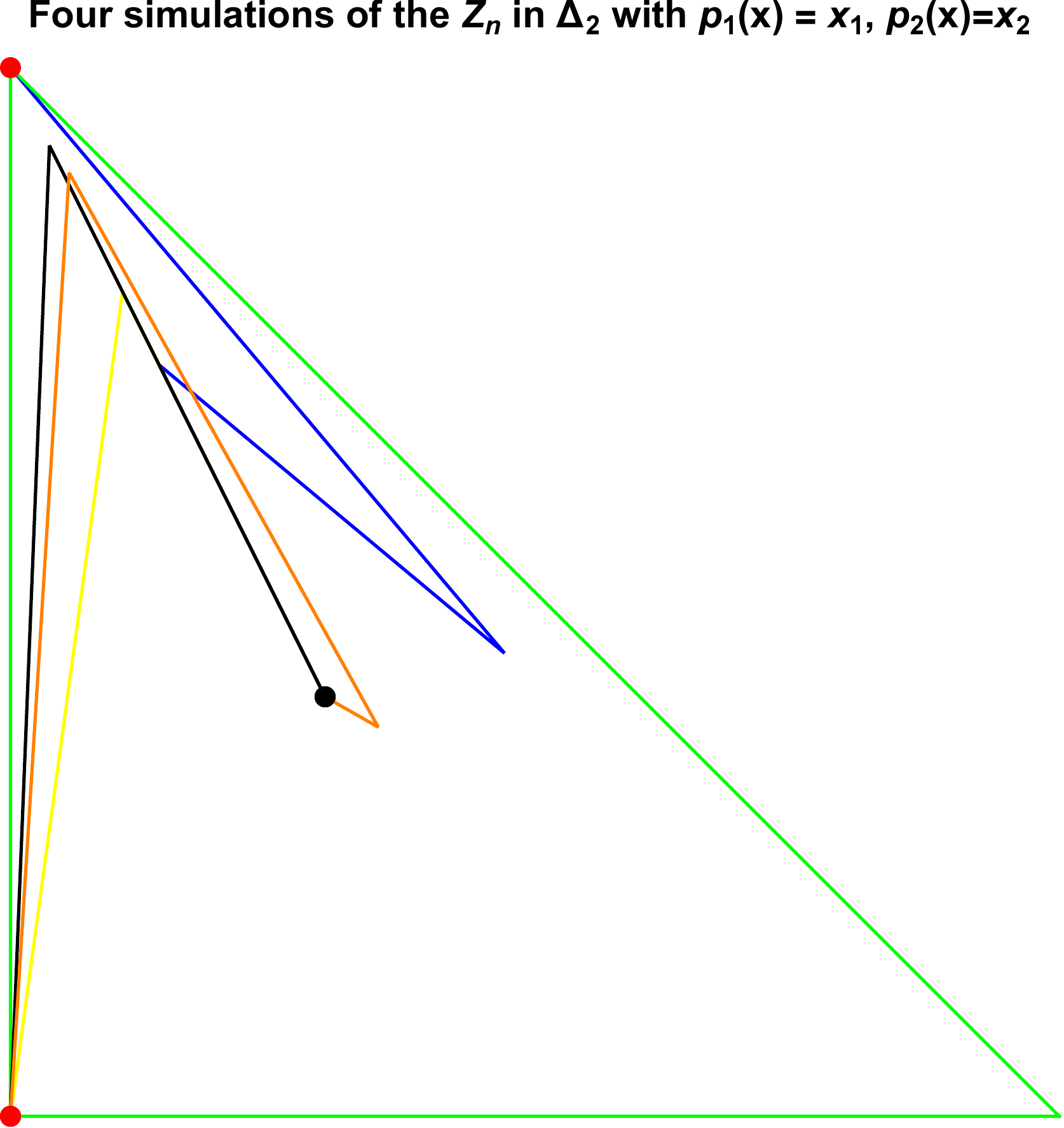}
	\end{center}
\caption{Left: Four random trajectories of $Z_n$ in $[0,1]$ starting at $x_0=0.8$, they will absorb in $\{0,1\}$; Right: Four random trajectories of $Z_n$ in $\De_2$ starting at $\x_0=(0.3,0.4)$, they will absorb in $\{\e_0, \e_2\}$.}
\label{Fig:2}
\end{figure}

\section{Asymptotic behavior of $(Z_n)_{n\ge 0}$}
In this section, we describe  the asymptotic behavior of $(Z_n)_{n\ge 0}$ using  the notion of {\it minimal $P$-absorbing compact subsets}. First we establish  some properties of the  family $\mathcal K_m$ of these subsets and propose their classification. By a general results of \cite{herve94}, this yields to the classification of   the set of $P$- invariant probability measures as well as the description of the asymptotic behavior of $(Z_n)_{n\ge 0}$. The classification is complete in $\De_2$ but partial in $\De_d, d>2$.

\subsection{The set $\mathcal K_m$ of minimal $P$-absorbing compact subsets}
\begin{defn}
A non-empty compact subset $K\subseteq \De_d$ is said to be {\it $P$-absorbing} if for all $\x \in K$
\[
P(\x,K^c) := P\dv_{K^c}(\x) = \suml_{i=0}^d p_i(\x) \int_0^1 \dv_{K^c}(t\x + (1-t)e_i)dt = 0,
\]
where $K^c = \De_d \setminus K$. It is {\it minimal}
when it does not contain any proper $P$-absorbing compact subset. 
\end{defn}
We denote by $\mathcal K_m$ is the set of all minimal $P$-absorbing compact subsets.
For any  $\x_0 \in \De_d$ and $\eps>0$, we set $B_{\eps}(\x_0) = \{\x \in \De_d: |\x-\x_0| < \eps\}$   and
$B_{\eps} = \cup_{i=0}^d B_{\eps}(\e_i)$. 
\begin{figure}[htb!]
\begin{center}
\begin{tikzpicture}[scale=1]
    \draw[color=red, ultra thick, ->] (0,0) -- (4,0);
    \draw[color=red, ultra thick,->] (0,0) -- (0,4);
       \draw[color=red, ultra thick] (3,0) -- (0,3);
                    \filldraw[fill=cyan, draw=blue] (0,0) -- (.4,0) arc (0:90:.4) -- (0,0);
                     \filldraw[fill=cyan, draw=blue] (3,0) -- (2.6,0) arc (180:135:.4) -- (3,0);

 \filldraw[fill=cyan, draw=blue] (0,3) -- (0,2.6) arc (-90:-45:.4) -- (0,3);
                                          \fill[blue] (0,0) circle (1.5pt);
                                           \fill[blue] (0,3) circle (1.5pt);
                                            \fill[blue] (3,0) circle (1.5pt);
                         
              \node[below] at (3,0) {$\e_1$} ; 
                \node[below] at (0,0) {$\e_0$} ; 
                         \node[left] at (0,3) {$\e_2$} ; 
                                \node[left] at (0,0.3) {$B_{\eps}(\e_0)$} ;
         \node[right] at (0.2,3) {$B_{\eps}(\e_2)$} ;
           \node[right] at (3,.3) {$B_{\eps}(\e_1)$} ;
                                                     
       \node[right] at (4,0) {$x_1$} ; 
              \node[left] at (0,4) {$x_2$} ; 
                         \end{tikzpicture}
                         
                         \end{center}
           \caption{Domain $B_{\eps}(\e_i)$}\label{fig:absorbingset1}
  \end{figure}
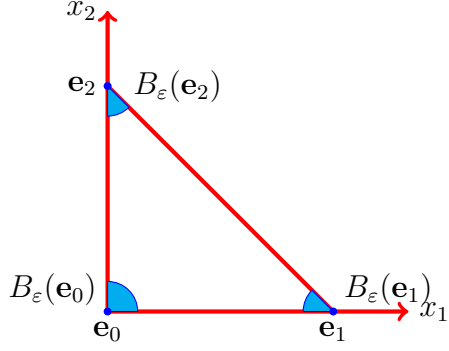
  
  The following  rules are useful to describe  the  minimal $P$-absorbing sets $K$.
\begin{prop}\label{prop:rules}
\begin{enumerate}
\item[(i)] \label{prop:atleast}
If $K\in \mathcal K_m$ then $K$ contains at least one vertex.
\item[(ii)] \label{prop:singular}
If $K\in \mathcal K_m$, $\e_i\in K$ and $p_i(\e_i) =1$ then $K=\{\e_i\}$.
\item[(iii)] \label{prop:halfsingular}
If $K\in \mathcal K_m$, $\e_i\in K$ and $p_j(\e_i) > 0$ for some $j\ne i$ then $[\e_i,\e_j] \subseteq K$.
\item[(iv)] \label{prop:fill}
If $K\in \mathcal K_m$ and ${p_i}(\x) > 0$ for some $\x \in K\setminus\{\e_i\}$ then $[\e_i,\x] \subseteq K$.
\end{enumerate}
\end{prop}
\begin{proof}
\begin{enumerate}
\item[(i)]
Assume that $\e_i \notin K$ for all $i=0,\ldots, d$. Since $K^c$ is open,  there exists $\eps > 0$ such that $B_{\eps} \subseteq K^c$. Therefore, for all $\x \in K$,
\begin{align*}
P(\x,K^c) &= \sum_{i=0}^d p_i(\x) \int_0^1 \dv_{K^c}(t\x +(1-t) e_i)dt \ge  \sum_{i=0}^d p_i(\x) \int_0^1 \dv_{B_{\eps}}(t\x +(1-t) e_i)dt\\
 &= \sum_{i=0}^d p_i(\x) \eps = \eps >0.
 \end{align*}
This contradicts to the fact that  $K$ is $P$-absorbing.

\item[(ii)] It suffices to prove that $\{\e_i\} \in \mathcal K_m$. This is true because 
\[
P(\e_i, \{\e_i\}^c) = \suml_{j\ne i} p_j(\e_i) = 0. 
\]
\item[(iii)] If $[\e_i,\e_j] \cap K^c\neq \emptyset$,  then there exist  $\x_0 \in [\e_i,\e_j] \cap K^c$ and  $\eps>0$ such that $B_{\eps}(\x_0) \subseteq K^c$. Therefore 
$
P(\e_i,K^c) \ge p_j(\e_i) \eps > 0
, $   contradiction.
\item[(iv)] Again, if $[\e_i,\e_j] \cap K^c\neq \emptyset$,  then there exist  $\x_0 \in [\e_i,\e_j] \cap K^c$ and  $\eps>0$ such that $B_{\eps}(\x_0) \subseteq K^c$; hence,
\[
P(\x,K^c) \ge p_i(\x) \eps > 0
\]
which is a contradiction.
\end{enumerate}
\end{proof}
This Proposition~\eqref{prop:rules}  easily yields to the classification of $\mathcal K_m$ when$d=1$ (see \cite{LP19}):
\begin{enumerate}
\item[(i)] If $p_1(1) <1 $ and $p_0(0) < 1$ then $\mathcal K_m = \{[0,1]\}$;
\item[(ii)] If $p_1(1) <1$ and $p_0(0) = 1$ then $\mathcal K_m = \{\{1\}\}$;
\item[(iii)] If $p_1(1) = 1$ and $p_0(0) < 1$ then $\mathcal K_m = \{\{0\}\}$;
\item[(iv)] If $p_1(1) = 1$ and $p_0(0) = 1$ then $\mathcal K_m = \{\{0\},\{1\}\}$.
\end{enumerate}
 
In  the  following section, we describe the set $\mathcal K_m$ in $\Delta_2$. Section  \ref{asymptoticZnDelta2} is devoted to the asymptotic behavior in distribution of $(Z_n)_{n \geq 0}$. The reader may be easily convinced  that similar statements hold in higher dimension.

\subsection{Classification of $\mathcal K_m$ in $\De_2$}
We assume $d=2$ in this subsection. Unlike the case $d=1$, for the case of $d=2$ we need to classify the values of $p_i$ not only on the vertices but also on the edges. We denote by $L_0=\{\x \in [\e_1,\e_2]: p_0(\x)>0\}$, $L_1=\{\x \in [\e_0,\e_2]: p_1(\x)>0\}$, and $L_2=\{\x \in [\e_0,\e_1]: p_2(\x)>0\}$. Let $L^c_0$ (resp. $L^c_1$ and  $L^c_2$) be the complement of $L_0$ (resp. $L_1$ and $L_2$) in $ [\e_1,\e_2]$ (resp.  $[\e_0,\e_2]$  and  $[\e_0,\e_1]$).

Let us fix $K \in \mathcal K_m$. There are several  cases to consider.
\begin{enumerate}
\item \underline{$p_0(\e_0) = p_1(\e_1)=p_2(\e_2)=1$}

By Proposition~\ref{prop:rules} [i], either $\e_0 \in K$ or $\e_1 \in K$ or $\e_2 \in K$. When  $\e_0\in K$, Proposition~\ref{prop:rules} [ii] implies $K=\{\e_0\}$; similarly for $\e_1$ or $\e_2$. Finally $\mathcal K_m = \{\{\e_0\},\{\e_1\},\{\e_2\}\}$.
\item  \underline{$p_0(\e_0) = p_1(\e_1)=1 $ but $p_2(\e_2)<1$}

As above, if $\e_0\in K$ (resp. $\e_1\in K$) then $K=\{\e_0\}$ (resp.  $\e_1\in K$).

Assume now $\e_2\in K$. Proposition~\ref{prop:rules} [iii] implies $[\e_0,\e_2] \subseteq K$ when $p_2(\e_0) > 0$ and   $[\e_1,\e_2] \subseteq K$ when $p_2(\e_1) > 0$. Therefore $K$ contains $\e_1$ or $\e_2$, contradiction with the minimality of $K$. Finally  $\mathcal K_m = \{\{\e_0\},\{\e_1\}\}$. 

Similar statements hold when $p_0(\e_0) = p_2(\e_2) =1$ but $p_1(\e_1) <1$ or $p_1(\e_1) = p_2(\e_2) = 1$ but $p_0(\e_0) <1$.
\item \underline{$p_0(\e_0) =1$ but $p_1(\e_1), p_2(\e_2)<1$}

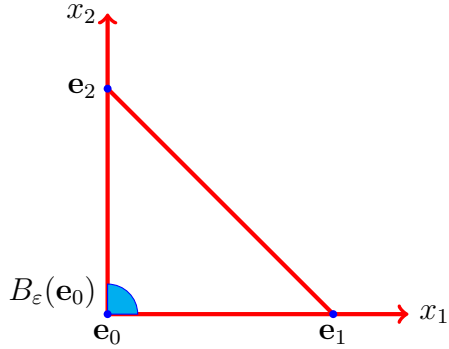
\begin{figure}[htb!]
\begin{center}
\begin{tikzpicture}[scale=1]
    \draw[color=red, ultra thick, ->] (0,0) -- (4,0);
    \draw[color=red, ultra thick,->] (0,0) -- (0,4);
       \draw[color=red, ultra thick] (3,0) -- (0,3);
                         \filldraw[fill=cyan, draw=blue] (0,0) -- (.4,0) arc (0:90:.4) -- (0,0);
                                                            \fill[blue] (0,0) circle (1.5pt);
                                           \fill[blue] (0,3) circle (1.5pt);
                                            \fill[blue] (3,0) circle (1.5pt);

              \node[below] at (3,0) {$\e_1$} ; 
                \node[below] at (0,0) {$\e_0$} ; 
                         \node[left] at (0,3) {$\e_2$} ; 
                                \node[left] at (0,0.3) {$B_{\eps}(\e_0)$} ;
                                                          
       \node[right] at (4,0) {$x_1$} ; 
              \node[left] at (0,4) {$x_2$} ; 
                         \end{tikzpicture}
                         \end{center}
           \caption{Domain \red{$B_{\eps}(\e_0)$}}\label{fig:absorbingset2}
  \end{figure}

Firstly, $\{\e_0\}\in \mathcal K_m$ and
 $K=\{\e_0\}$ as soon as $\e_0\in K$.

Assume now $\e_0\notin K$ (thus $B_{\eps}(\e_0) \subseteq K^c$ for some $\eps>0$) and suppose  for instance $\e_1 \in K$ (the same argument holds with $\e_2$). Hence,  $p_0(\e_1) = 0$; indeed, condition $p_0(\e_1) > 0$  implies $[\e_0,\e_1] \subseteq K$, contradiction. Consequently  $p_2(\e_1) > 0$, which implies  $[\e_1,\e_2] \subseteq K$, then $p_0(\e_2) = 0$ and $p_1(\e_2) > 0$. This readily implies that $L_0 =\emptyset$;  otherwise,  $p_0(\x_0) > 0$ for some $\x_0\in [\e_1,\e_2]$,   therefore 
$
P(\x_0, K^c) \ge p_0(\x_0) \eps >0, 
$
 contradiction with the fact that $\x_0\in K$ and $K$ is invariant. The equality $L_0 =\emptyset$ yields $K=[\e_1,\e_2]$  and  $\mathcal K_m = \{\{\e_0\},[\e_1,\e_2]\}$.

Finally
\[\mathcal K_m = 
\begin{cases}
\{\{\e_0\}, [\e_1,\e_2]\}, & \text{ if }  L_0=\emptyset\: \\
\{\{\e_0\}\}, & \text{ else. }
\end{cases}
\]
Similar statements hold when $p_1(\e_1)=1$ but $p_0(\e_0), p_2(\e_2)<1$ or $p_2(\e_2) =1$ but $p_0(\e_0), p_1(\e_1)<1$.

\item \underline{$p_0(\e_0), p_1(\e_1), p_2(\e_2) <1$ and $L_0=\emptyset$}

By Proposition~\ref{prop:rules} [i], the set $K$ contains at least one of the vertices.  Assume for instance $\e_1 \in K$, thus  $p_2(\e_1) > 0 $ since $L_0=\emptyset$, which implies $[\e_1,\e_2] \subseteq K$. The condition $L_0=\emptyset$ also 
implies $
P(\x, [\e_1,\e_2]^c) = p_0(\x) = 0  
$
for all $\x\in[\e_1,\e_2]$, finally $K=[\e_1,\e_2]$. The same conclusion holds when $\e_2 \in K$.

Now,   the  set $K$ cannot contain $\e_0$. Otherwise, the condition $p_0(\e_0)<1$ implies either $[\e_0, \e_1] \subseteq K$ or $[\e_0, \e_2] \subseteq K$, thus $\e_1\in K$ or $\e_2\in K$. This yields $K=[\e_1, \e_2]$, contradiction.

Similar statements hold when $L_1=\emptyset$ or $L_2=\emptyset$.

\item  \underline{$p_0(\e_0), p_1(\e_1), p_2(\e_2) <1$ and $L_0, L_1, L_2$ are nonempty}

 In this case we always have $\{\e_0,\e_1,\e_2\} \subseteq K$. Indeed, by Proposition~\ref{prop:rules} [i], the set $K$ contains at least one vertex, say $\e_0 \in K$; since $p_0(\e_0) <1$, it contains even one of the two sides   $[\e_0,\e_1] $ or $[\e_0,\e_2] $.  Assume $[\e_0,\e_1] \subset K$  (thus $\e_1\in K$) and let us check that  $\e_2 \in K$. Otherwise $B_{\eps}(\e_2) \subseteq K^c$ for some $\eps>0$; since  $L_2$ is a proper subset of $[\e_0,\e_1]$,   there exists  $\x\in  [\e_0,\e_1] \subseteq K$ such that
$
P(\x,K^c) \ge p_2(\x)\eps >0
$,  contradiction.

Now, the inclusion $\{\e_0,\e_1,\e_2\} \subseteq K$ combined with Proposition~\ref{prop:rules} [iv]  yields 
\[co\{\e_0, L_0\} \subseteq K, co\{\e_1, L_1\} \subseteq K\quad \text{and} \quad co\{\e_2, L_2\} \subseteq K\]
So that, by the compactness of $K$, 
\bel{K0}
K_0 := \overline{co\{\e_0, L_0\} \cup co\{\e_1, L_1\} \cup co\{\e_2, L_2\}} \subseteq K.
\qe
Next, we denote by 
\bel{cond:last}
L_i(1) := \{\x \in K_0: p_i(\x) > 0\} ,\qquad \forall i\in \{0,1,2\}. 
\qe
It is easy to see that $L_i(1) \supseteq L_i$ for all $i\in \{0,1,2\}$. Then,  Proposition~\ref{prop:rules} [iv] and the compactness of $K$ yield
\bel{K1}
K_1 := \overline{co\{\e_0, L_0(1)\} \cup co\{\e_1, L_1(1)\} \cup co\{\e_2, L_2(1)\}} \subseteq K.
\qe
We iteratively construct a sequence of increasing compact subsets $\{K_n\}_{n\ge 0} \subseteq K$ and consider two possibilities: if there is a finite $n$ such that $K_n = \De_2$ then $K =\De_2$; otherwise, $K=K_{\infty}:=\overline{\{\cup_{n\ge 0} K_n\}} \subset \De_2$. In this case $\mathcal K_m$ consists of a unique minimal $P$-absorbing compact set
\[
\mathcal K_m=
\begin{cases}
\{\De_2\}, &\text{ if there exists a finite $n$ such that $K_n=\De_2$}, \\
\{K_{\infty}\}, &\text{ otherwise.}
\end{cases}
\] 
We illustrate here two cases when $\mathcal K_m = \{K_0\}$ and $\mathcal K_m = \{K_1\}$. 
\begin{enumerate}
\item We assume that, for $i=0, 1, 2$, 
\bel{cond:K_0only}
p_i(\x) = 0 \text{ for all } \x \in \Om_i:=\{\y\in K_0: [\y,\e_i] \cap K_0^c \ne \emptyset\}
\qe
Then  $K_1=K_0$ and, iteratively, $K_{\infty} = K_0$. Moreover, from \eqref{cond:K_0only},  $P(\x,K_0^c) = 0$ for all $\x\in K_0$, therefore  $K = K_0$ by the minimality of $K$ and \eqref{K0}.  
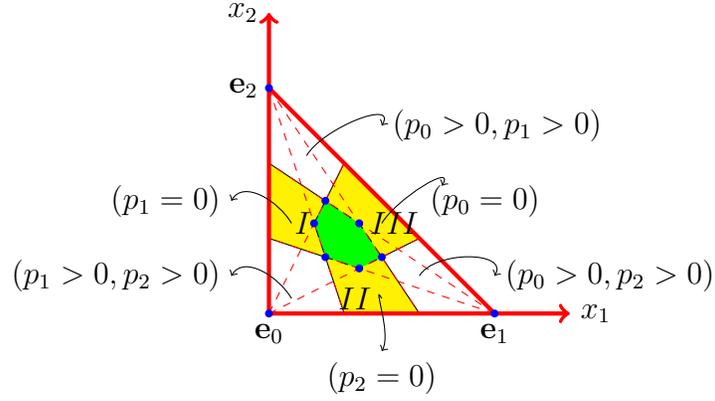
\begin{figure}[htb!]
\begin{center}
\begin{tikzpicture}[scale=1]
     \draw[fill=yellow]  (3/4,3/4) -- (1,0) -- (2,0) -- (3/2,3/4) -- (3/4,3/4);
       \draw[fill=yellow]  (3/2,3/4) -- (2,1) -- (1,2) -- (3/4,3/2) -- (3/2,3/4);
         \draw[fill=yellow]  (3/4,3/4) -- (3/4,3/2) -- (0,2) -- (0,1) -- (3/4,3/4);
             \draw[fill=green, dashed]  (3/4,3/4) -- (1.2, 0.6) -- (3/2,3/4) -- (1.2,1.2)-- (3/4,3/2) -- (0.6,1.2) -- (3/4,3/4);

    \draw[color=red, ultra thick, ->] (0,0) -- (4,0);
    \draw[color=red, ultra thick,->] (0,0) -- (0,4);
       \draw[color=red, ultra thick] (3,0) -- (0,3);
        \draw[color=red, dashed] (3,0) -- (0,1);
          \draw[color=red, dashed] (3,0) -- (0,2);
            \draw[color=red, dashed] (0,3) -- (1,0);
              \draw[color=red, dashed] (0,3) -- (2,0);
                \draw[color=red, dashed] (0,0) -- (2,1);
                  \draw[color=red, dashed] (0,0) -- (1,2);
                  
                     \node[left] at (1.5, .2) {${II}$} ; 
                        \node[right] at (0.2, 1.2) {$I$} ; 
                            \node[right] at (1.2, 1.2) {$III$} ; 
                     
                                                            \fill[blue] (0,0) circle (1.5pt);
                                           \fill[blue] (0,3) circle (1.5pt);
                                            \fill[blue] (3,0) circle (1.5pt);
                               \fill[blue] (3/4,3/4) circle (1.5pt);
      \fill[blue] (0.6,1.2) circle (1.5pt);

      \fill[blue] (1.2,0.6) circle (1.5pt);

      \fill[blue] (1.2,1.2) circle (1.5pt);

 \fill[blue] (3/2,3/4) circle (1.5pt);
 \fill[blue] (3/4,3/2) circle (1.5pt);
 
                 \node[right] at (2,1.5) {$(p_0 = 0)$} ; 
\draw[->] (1.5, 1.2)  to [out=60,in=60] (2.3,1.7);

          \node[below] at (1.5,-.5) {$(p_2 = 0)$} ; 
\draw[->] (1.5, .2)  to [out=120,in=60] (1.5,-.5);

          \node[left] at (-0.5,1.5) {$(p_1 = 0)$} ; 
\draw[->] (0.3, 1.2)  to [out=120,in=60] (-0.5,1.5);

   \node[left] at (-0.5,0.5) {$(p_1 > 0, p_2 >0)$} ; 
\draw[->] (0.3, 0.2)  to [out=120,in=60] (-0.5,0.5);

 \node[right] at (3,0.5) {$(p_0 > 0, p_2>0)$} ; 
\draw[->] (2, 0.5)  to [out=60,in=60] (3,0.5);

 \node[right] at (1.5,2.5) {$(p_0 > 0, p_1>0)$} ; 
\draw[->] (0.5, 2.1)  to [out=60,in=60] (1.5,2.5);

                                                    \node[below] at (3,0) {$\e_1$} ; 
                \node[below] at (0,0) {$\e_0$} ; 
                         \node[left] at (0,3) {$\e_2$} ; 
                                                                                       
       \node[right] at (4,0) {$x_1$} ; 
              \node[left] at (0,4) {$x_2$} ; 
                         \end{tikzpicture}
                         \end{center}
           \caption{$K_0=$ white domains $\cup $ yellow domains = $\De_2 \setminus$ green domain; $K_0^c$ is the green domain; $\Om_1, \Om_2, \Om_0$ are yellow domains $I, II, III$ correspondingly.}\label{fig:absorbingset3}
  \end{figure}
\item  We assume now, for $i=0, 1, 2,$ 
\bel{cond:K_1only}
p_i(\x) = 0 \text{ for all } \x \in \Om_i(1):=\{ i^{th}\  \text{yellow region}.\}
\qe
Then, $K_1\ne K_0$, $K_2=K_1$ and iteratively $K_{\infty} = K_1$. Moreover, by \eqref{cond:K_1only},   $P(\x,K_1^c) = 0$ for all $\x\in K_1$, therefore $K = K_1$ by the minimality of $K$ and \eqref{K1}.  

\begin{figure}[htb!]
\begin{center}
\begin{tikzpicture}[scale=1]
    \draw[fill=yellow]  (1,.67) -- (1,0) -- (2,0) -- (3/2,3/4) -- (1.22,1) -- (.95,.78) -- (1,.67);
       \draw[fill=yellow]  (1.35,1.11) -- (2,1) -- (1,2) -- (3/4,3/2) -- (.75,1.25)-- (1.22,1)--(1.35,1.11);
         \draw[fill=yellow]  (.65,1.315) -- (0,2) -- (0,1) -- (3/4,3/4) -- (.95,.78) -- (.75,1.25) -- (.65,1.315);
             \draw[fill=green, dashed]  (.95,.78)--(1.22,1)--(.75,1.25)--(.95,.78);

    \draw[color=red, ultra thick, ->] (0,0) -- (4,0);
    \draw[color=red, ultra thick,->] (0,0) -- (0,4);
       \draw[color=red, ultra thick] (3,0) -- (0,3);
        \draw[color=red, dashed] (3,0) -- (0,1);
          \draw[color=red, dashed] (3,0) -- (0,2);
            \draw[color=red, dashed] (0,3) -- (1,0);
              \draw[color=red, dashed] (0,3) -- (2,0);
                \draw[color=red, dashed] (0,0) -- (2,1);
                  \draw[color=red, dashed] (0,0) -- (1,2);
                                 \draw[color=blue, dotted] (0,0) -- (1.35,1.11);
                                    \draw[color=blue, dotted] (0,3) -- (1,.67);
                                      \draw[color=blue, dotted] (3,0) -- (.65,1.315);
                         \fill[blue] (1,.67) circle (.5pt);
                               \fill[blue] (1.35,1.11) circle (.5pt);
                               \fill[blue] (.65,1.315) circle (.5pt);
                     \node[left] at (1.5, .2) {${II}$} ; 
                        \node[right] at (0.2, 1.2) {$I$} ; 
                            \node[right] at (1.2, 1.2) {$III$} ; 
                            
                                     \fill[blue] (.95,.78) circle (1pt);
                                                 \fill[blue] (1.22,1) circle (1pt);
                                                             \fill[blue] (.75,1.25) circle (1pt);
                     
                                                            \fill[blue] (0,0) circle (1.5pt);
                                           \fill[blue] (0,3) circle (1.5pt);
                                            \fill[blue] (3,0) circle (1.5pt);
                               \fill[blue] (3/4,3/4) circle (.5pt);
      \fill[blue] (0.6,1.2) circle (.5pt);

      \fill[blue] (1.2,0.6) circle (.5pt);

      \fill[blue] (1.2,1.2) circle (.5pt);

 \fill[blue] (3/2,3/4) circle (.5pt);
 \fill[blue] (3/4,3/2) circle (.5pt);
 
                 \node[right] at (2,1.5) {$(p_0 = 0)$} ; 
\draw[->] (1.5, 1.2)  to [out=60,in=60] (2.3,1.7);

          \node[below] at (1.5,-.5) {$(p_2 = 0)$} ; 
\draw[->] (1.5, .2)  to [out=120,in=60] (1.5,-.5);

          \node[left] at (-0.5,1.5) {$(p_1 = 0)$} ; 
\draw[->] (0.3, 1.2)  to [out=120,in=60] (-0.5,1.5);

   \node[left] at (-0.5,0.5) {$(p_1 > 0, p_2 >0)$} ; 
\draw[->] (0.3, 0.2)  to [out=120,in=60] (-0.5,0.5);

 \node[right] at (3,0.5) {$(p_0 > 0, p_2>0)$} ; 
\draw[->] (2, 0.5)  to [out=60,in=60] (3,0.5);

 \node[right] at (1.5,2.5) {$(p_0 > 0, p_1>0)$} ; 
\draw[->] (0.5, 2.1)  to [out=60,in=60] (1.5,2.5);
                                                    \node[below] at (3,0) {$\e_1$} ; 
                \node[below] at (0,0) {$\e_0$} ; 
                         \node[left] at (0,3) {$\e_2$} ; 
                                                                                       
       \node[right] at (4,0) {$x_1$} ; 
              \node[left] at (0,4) {$x_2$} ; 
                          \end{tikzpicture}
                         \end{center}
           \caption{$K_1=$ white domains $+$ yellow domains = $\De_2$ - green domain; $K_1^c$ is green domain; $\Om_1(1), \Om_2(1), \Om_0(1)$ are yellow domains $I, II, III$ correspondingly.}\label{fig:absorbingset4}
  \end{figure}
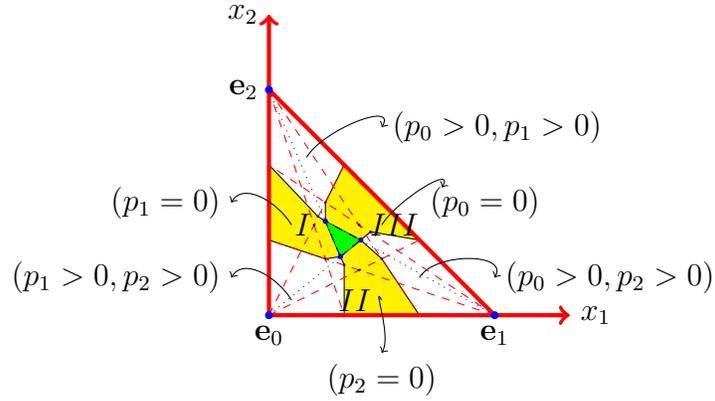
\end{enumerate}

\end{enumerate}

\newpage 
In summary, the complete classification of $\mathcal K_m$ in $\De_2$ is as follows:
\begin{thm} \label{classdelta2} In $\De_2$,  
\begin{enumerate}
\item either $\mathcal K_m$ consists of $3$ vertices;
\item  or $\mathcal K_m$ consists of $2$ vertices;
\item or $\mathcal K_m$ consists of $1$ vertex;
\item or $\mathcal K_m$ consists of $1$ edge;
\item or $\mathcal K_m$ consists of $1$ vertex and $1$ opposite edge;
\item or $\mathcal K_m$ consists of a compact subset $K_\infty \subseteq \De_2$ such that $K_\infty \cap \mathring{\De}_2\neq \emptyset.$ This set $K_\infty$ may equal the whole set $\De_2$.
\end{enumerate}
\end{thm}

\subsection{Asymptotic behavior of $(Z_n)_{n\ge 0}$ in $\Delta_2$}\label{asymptoticZnDelta2}

Using  Theorem \ref{classdelta2}, we may state the following Theorem
in $\De_2$.
\begin{thm}
\label{2DimGENERALDF}
 Let  $(Z_n)_{n \geq 0}$ be the Diaconis and Freedman's chain in $\De_2$ with weight functions $p_i(\x) \in \mathcal H_\alpha (\De_2)$. Denote by $\mathcal I (P)$ the set of the invariant  probability measures of $(Z_n)_{n \geq 0}$. Then, one of the following options holds.
\begin{enumerate}
\item If $\mathcal K_m = \{\{\e_0\}, \{\e_1\},\{\e_2\}\}$ then $\mathcal I (P) = co\{\de_{\e_0},\de_{\e_1},\de_{\e_2}\}$ and for any $\x \in \De_2$, the chain $(Z_n)_{n\geq 0}$ converges $\PP_{\x}$-a.s. to a random variable $Z_{\infty}$ with values in $\{\e_0,\e_1,\e_2\}$ and distribution \[
\PP_{\x}(Z_{\infty} = \e_i) =h_i(\x), \quad i\in \{0,1,2\}
\]
where $h_i$ is a nonnegative function in $H_{\al}(\De_2)$ such that $Ph_i=h_i$, $h_0+h_1+h_2 \equiv 1$, and
      $h_i(\e_j)=0$ for all $i\ne j\in \{0,1,2\}$.
Moreover, there exist $\kappa >0$ and $\rho\in [0, 1)$ such that
$$ \forall \varphi \in \mathcal   H_\alpha(\De_2), \ \forall \x \in \De_2 \quad \Bigg\vert P^n\varphi(\x)-h_0(\x)\varphi(\e_0)-h_1(\x)\varphi(\e_1)-h_2(\x)\varphi(\e_2)\Bigg\vert \leq  \kappa \rho^n \Vert \varphi\Vert_\alpha.$$

\item If $\mathcal K_m = \{\{\e_0\}, \{\e_1\}\}$ then $\mathcal I (P) = co\{\de_{\e_0},\de_{\e_1}\}$ and for any $\x \in \De_2$, the chain $(Z_n)_{n\geq 0}$ converges $\PP_{\x}$-a.s. to a random variable $Z_{\infty}$ with values in $\{\e_0,\e_1\}$ and distribution
\[
\PP_{\x}(Z_{\infty} = \e_i) =h_i(\x),  \quad i\in \{0,1\}
\]
where $h_i$ is the unique function in $H_{\al}(\De_2)$ such that $Ph_i=h_i$, $h_0+h_1\equiv 1$, and
      $h_i(\e_j)=\de_{ij}$ for all $i,j\in \{0,1\}$.
Moreover, there exist $\kappa >0$ and $\rho\in [0, 1)$ such that
$$ \forall \varphi \in \mathcal   H_\alpha(\De_2), \ \forall \x \in \De_2 \quad \Bigg\vert P^n\varphi(\x)-h_0(\x)\varphi(\e_0)-h_1(\x)\varphi(\e_1)\Bigg\vert \leq  \kappa \rho^n \Vert \varphi\Vert_\alpha.$$ Similar statements hold when $\mathcal K_m = \{\{\e_0\}, \{\e_2\}\}$ or $\mathcal K_m = \{\{\e_1\}, \{\e_2\}\}$.

\item If $\mathcal K_m = \{\{\e_0\}\}$ then $\mathcal I (P)=\{\de_{\e_0}\}$ and for any $\x \in \De_2$, the chain $(Z_n)_{n\geq 0}$ converges $\PP_{\x}$-a.s. to $\e_0$. Moreover, there exist $\kappa >0$ and $\rho\in [0, 1)$ such that
$$ \forall \varphi \in \mathcal   H_\alpha(\De_2), \ \forall \x \in \De_2 \quad \Bigg\vert P^n\varphi(\x)-\varphi(\e_0)\Bigg\vert \leq  \kappa \rho^n \Vert \varphi\Vert_\alpha.$$ Similar statements hold when $\mathcal K_m = \{\{\e_1\}\}$ or $\mathcal K_m = \{\{\e_2\}\}$.

\item If $\mathcal K_m = \{[\e_1,\e_2]\}$ then $\mathcal I (P) = \{\mu_{\infty}^{12}(d\x \cap [\e_1,\e_2])\}$ where $\mu_{\infty}^{12}$ is the  probability measure on $[\e_1,\e_2]$ with density
\[
g_{\infty}^{12}((t,1-t),(s,1-s)) := C\exp\left(\int_s^{t} \frac{p_1(u,1-u)}{1-u} du + \int_{t}^s \frac{p_2(u,1-u)}{u} du\right).
\] 
For any $\x \in \De_2$, the chain $(Z_n)_{n\geq 0}$ converges $\PP_{\x}$-a.s. to a random variable $Z_{\infty}$ with values on $[\e_1,\e_2]$ and distribution $\mu_{\infty}^{12}(d\x \cap [\e_1,\e_2])$. Moreover, there exist $\kappa >0$ and $\rho\in [0, 1)$ such that
$$ \forall \varphi \in \mathcal   H_\alpha(\De_d), \ \forall \x \in \De_d \quad \Bigg\vert P^n\varphi(\x)-
\mu_\infty^{12}(\varphi)
\Bigg\vert \leq  \kappa \rho^n \Vert \varphi\Vert_\alpha.$$
Similar statements hold when $\mathcal K_m = \{[\e_0,\e_1]\}$ or $\mathcal K_m = \{[\e_0,\e_2]\}$.

\item If $\mathcal K_m = \{\{\e_0\}, [\e_1,\e_2]\}$  then $\mathcal I (P) = co\{\de_{\e_0},\chi_{[\e_1,\e_2]}\}$  and for any $\x \in \De_2$, the chain $(Z_n)_{n\geq 0}$ converges to $\e_0$ with probability $h_0(\x)$ and to $ [\e_1,\e_2]$ with probability $h_{12}(\x)$. Moreover, there exist $\kappa >0$ and $\rho\in [0, 1)$ such that $\: \forall \varphi \in \mathcal   H_\alpha(\De_2), \forall \x \in \De_2 $
\[ \Bigg\vert P^n\varphi(\x)-h_0(\x)\varphi(\e_0) - h_{12}(\x)
\mu_\infty^{12}(\varphi)
\Bigg\vert \leq  \kappa \rho^n \Vert \varphi\Vert_\alpha.\]

\item If $\mathcal K_m = \{K_{\infty}\}$ (possible equal $\De_2$) then $\mathcal I (P) = \{\mu_{\infty}\}$ which is a probability measure on $\De_2$ with support $K_{\infty}$ (possible equal $\De_2$). 
\end{enumerate}
 
 \end{thm}
 \begin{proof}
 By a directed calculation, for all $\varphi\in H_{\al}(\De_2)$,
\[
\vert P\varphi\vert _\al \le \frac{1}{1+\al}\vert \varphi\vert _\al + \Bigg(1+\suml_{i=0}^2 m_{\al}(p_i)\Bigg)|\varphi|_{\infty}.
\]
Hence, by  \cite{Hennion93}, the operator $P$ is quasi-compact on $H_\alpha (\De_2)$. The operator $P$ is Markov, so that 
 $\dv \in H_{\al}(\De_2)$  satisfies $P\dv = \dv$. Therefore, by using the \cite[Theorem 2.2]{herve94}, the eigenspace corresponding to eigenvalue $1$ is nothing but $\ker (P-Id)$. All six above cases can be checked easily by following the \cite{herve94} (also see in \cite{LP19} for a classification in  dimension 1).

 \end{proof}

\begin{rem}
The cases considered in Section 4 where there is a unique invariant probability density all satisfy the case 6 where $\mathcal K_m = \{\De_2\}$, i.e. when $p_i(\e_i)<1$ for all $i=0,1,2$. The question of the existence (hence unicity) of an invariant probability density when $p_i(\e_i)<1$ for all $i=0,1,\ldots,d$  is still open for $d\ge 2$ (it has been solved in $d=1$ in \cite{LP19}).
\end{rem}

 \section{Discussion}

We would like to briefly present here another interesting setting for the Diaconis and Freedman's chain in $\De_d$. 
For any $i=0, \ldots, d$ and $\x \in \Delta_d$,  let $S_i(\x)$ be the  strict convex combination of $\x$ and all vertices $\e_j$ except $\e_i$, i.e. $S_i(\x) := \mathring{\Big(co \{\x, \{\e_j\}_{j\ne i}\}\Big)}.$

Assume that at time $n$, a walker $\Z$  is located at site $Z_n=\x\in \De_d$ and has probability $p_i(\x) $ to move to the domain $S_i(\x)$, the arrival point being    chosen  according to the  uniform distribution  $ U_{S_i}(d\x))$ on this domain.    In other words, the one-step transition probability function of the Markov chain generated by this walker is 
\[
P(\x,d\y) = \sum_{i=0}^d p_i(\x) \frac{1}{|S_i(\x)|}\dv_{S_i(\x)}(\y) d\y, \quad \x \in \mathring{\De}_d.
\]
We illustrate this setting in $\De_2$ in Figure~\ref{fig:2} but a such setting and its applications will be considered in details somewhere else.
\begin{figure}[htb!]
\begin{center}
\begin{tikzpicture}[scale=1.2]
    \draw[color=red, ultra thick] (0,0) -- (3,0);
    \draw[color=red, ultra thick] (0,0) -- (0,3);
    \draw[color=red, ultra thick] (0,3) -- (3,0);
    \draw[color=red, dashed] (0,0) -- (1,1);
    \draw[color=red, dashed] (0,3) -- (1,1);
    \draw[color=red, dashed] (1,1) -- (3,0);
   \node[left] at (0,0) {$e_0$} ; 
   \node[right] at (3,0) {$e_1$} ; 
   \node[left] at (0,3) {$e_2$} ; 
   \node[below,right] at (1,1) {$x$} ; 
   \fill[blue] (0,0) circle (1.5pt);
   \fill[blue] (1,1) circle (1.5pt);   
   \fill[blue] (3,0) circle (1.5pt);
    \fill[blue] (0,3) circle (1.5pt);
    \draw[color=green, ultra thick, ->] (1,1) .. controls (0.7,.7) .. (0.6,1);
     \draw[color=green, ultra thick, ->] (1,1) .. controls (1.1,.5)  .. (1.5,.5);
      \draw[color=green, ultra thick, ->] (1,1) .. controls (1.1,1.1)  .. (1.4,1.2);
     \node[right] at (1.5,0.5) {$p_2(x)$} ; 
      \node[left] at (0.6,1) {$p_1(x)$} ; 
            \node[right] at (1.4,1.2) {$p_0(x)$} ; 
          \node[left] at (0.7,1.5) {$S_1(x)$} ; 
               \node[right] at (1.5,0.2) {$S_2(x)$} ; 
                \node[] at (1.1,1.5) {$S_0(x)$} ; 
         \end{tikzpicture}
\end{center}
\caption{ An alternative model
}
\label{fig:2}
\end{figure}
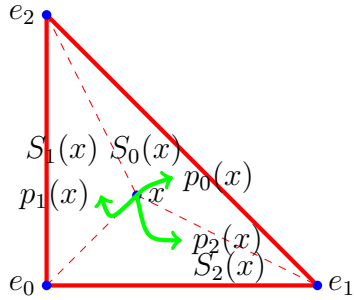

\section*{Acknowledgments}

M. Peign\'e and T.D. Tran would like to thank warmly VIASM for financial support and hospitality where the first ideas of paper has been discussed: M. Peign\'e spent 6 months in 2017 and T.D. Tran spent three months in 2017 and two months in 2019 at VIASM, as visiting scientists. T.D. Tran would also like to thank Institut Denis Poisson for financial support and a warmly and friendly hospitality during his one-month visiting in 2018. We would like to thank J\"{u}rgen Jost for illuminating discussions.

\bibliographystyle{apalike}


\Addresses

\end{document}